\documentclass[oneside,11pt,reqno]{amsart}
\usepackage[dvipsnames]{xcolor}
\usepackage{graphicx}
\usepackage{tikz-cd}
\usepackage{amsmath, amssymb,amsthm,mathrsfs}
\usepackage{color}
\usepackage{mathtools}
\usepackage{nameref}
\usepackage{url}
\usepackage[colorlinks=true,linkcolor=blue!60!black,citecolor=teal!80!black,urlcolor=RoyalBlue]{hyperref}
\usepackage{cleveref}
\usepackage[margin=2.8cm]{geometry}
\usepackage{accents}
\usepackage{import}
\usepackage{enumerate}
\usepackage{tikz}
\usepackage{xspace}

\usepackage[T1]{fontenc}
\usepackage{lmodern}

\newtheorem{thm}{Theorem}[section]
\newtheorem{cor}[thm]{Corollary}
\newtheorem{lemma}[thm]{Lemma}
\newtheorem{prop}[thm]{Proposition}

\newtheorem*{thmA}{Theorem A}
\newtheorem*{thmB}{Theorem B}
\newtheorem*{thmC}{Theorem C}
\newtheorem*{thmD}{Theorem D}

\newcommand{\thmAref}{\hyperlink{thm:A}{Theorem~A}\xspace}
\newcommand{\thmBref}{\hyperlink{thm:B}{Theorem~B}\xspace}
\newcommand{\thmCref}{\hyperlink{thm:C}{Theorem~C}\xspace}
\newcommand{\thmDref}{\hyperlink{thm:D}{Theorem~D}\xspace}
  
\newtheorem{dfn&prop}[thm]{Definition and Proposition}
\newtheorem{dfn&thm}[thm]{Definition and Theorem}

\theoremstyle{definition}
\newtheorem{defn}[thm]{Definition}

\newtheorem{observation}[thm]{Observation}

\newtheorem*{structure}{Structure of the article}
\newtheorem{rem}[thm]{Remark}

\theoremstyle{remark}
\newtheorem*{Claim}{Claim}

\newtheorem*{remark}{Remark}

\numberwithin{equation}{section}

\newenvironment{subproof}{\begin{proof}[Proof of claim.]}{%
	\end{proof}}

\DeclareMathOperator{\Mod}{mod}
\DeclareMathOperator{\Per}{Per}
\DeclareMathOperator{\Crit}{Crit}
\DeclareMathOperator{\GO}{GO}

\DeclareMathOperator{\diam}{diam}

\def\phi{\varphi}

\def\F{{\mathcal{F}}}

\newcommand{\A}{\mathbb{A}}

\newcommand{\s}{\mathbb{S}}
\newcommand{\CC}{\widehat{\mathbb{C}}}
\def\chat{\widehat{\mathbb{C}}}
\def\C{{\mathbb{C}}}
\def\D{{\mathbb{D}}}
\def\N{{\mathbb{N}}}
\def\Z{{\mathbb{Z}}}
\def\R{{\mathbb{R}}}

\newcommand{\cl}{\overline}

\newcommand{\Orb}{\operatorname{GO}}
\newcommand{\la}{\lambda}

\newcommand{\defmap}[5]{
  \begin{array}{rccc}
    #1: & #2 &\longrightarrow &#3\\
        & #4 &\longmapsto & #5
  \end{array}
}

\DeclareMathOperator{\interior}{int}

\newcommand*{\defeq}{\mathrel{\vcenter{\baselineskip0.5ex \lineskiplimit0pt
			\hbox{\scriptsize.}\hbox{\scriptsize.}}}%
	=}
\newcommand{\eqdef}{=\mathrel{\vcenter{\baselineskip0.5ex \lineskiplimit0pt
			\hbox{\scriptsize.}\hbox{\scriptsize.}}}}

\title{Grand orbit relations in wandering domains}

\author[V. Evdoridou]{Vasiliki Evdoridou}
\author[N. Fagella]{N\'uria Fagella$^*$}
\author[L. Geyer]{Lukas Geyer}
\author[L. Pardo-Simón]{Leticia Pardo-Sim\'on}

\address{\noindent School of Mathematics and Statistics \\The Open University\\ Walton Hall\\ Milton Keynes MK7 6AA\\
	UK \\ \textsc{\newline \indent \href{https://orcid.org/0000-0002-5409-2663}{\includegraphics[width=1em,height=1em]{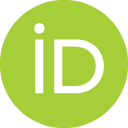} {\normalfont https://orcid.org/0000-0002-5409-2663}}}}
\email{vasiliki.evdoridou@open.ac.uk}

\address{\noindent Dep. de Matemàtiques i Informàtica\\ Universitat de Barcelona\\ Catalonia\\ Spain\\
	\newline  Centre de Recerca Matemàtica\\ Bellaterra\\ Catalonia\\ Spain.
	\textsc{\newline \indent 
		\href{https://orcid.org/0000-0002-5466-0579%
		}{\includegraphics[width=1em,height=1em]{orcid2.png} {\normalfont https://orcid.org/0000-0002-5466-0579}}
}}
\email{nfagella@ub.edu}

\address{\noindent Dept. of Mathematical Sciences\\ Montana State University, Bozeman\\ MT 59717–2400\\ USA
	\textsc{\newline \indent 
		\href{https://orcid.org/0000-0001-5889-9037%
		}{\includegraphics[width=1em,height=1em]{orcid2.png} {\normalfont https://orcid.org/0000-0001-5889-9037}}
}}
\email{geyer@montana.edu}

\address{\noindent Dep. de Matemàtiques i Informàtica\\ Universitat de Barcelona\\ Catalonia\\ Spain\\
	\textsc{\newline \indent 
		\href{https://orcid.org/0000-0003-4039-5556%
		}{\includegraphics[width=1em,height=1em]{orcid2.png} {\normalfont https://orcid.org/0000-0003-4039-5556}}
}}
\email{lpardosimon@ub.edu}

\thanks{
  domains, Teichmüller space, quasiconformal surgery.\\ *Supported by the Spanish State Research Agency through the MdM grant CEX2020-001084-M and grant PID2020-118281GB-C32; and by the Catalan government through ICREA Academia 2020.
  }

\begin{document}

\begin{abstract} One of the fundamental distinctions in McMullen and
  Sullivan's description of the Teichmüller space of a complex
  dynamical system is between discrete and indiscrete grand orbit
  relations. We investigate these on the Fatou set of transcendental
  entire maps and provide criteria to distinguish between the two
  types.  Furthermore, we show that discrete and indiscrete grand
  orbit relations may coexist non-trivially in a wandering domain, a
  phenomenon which does not occur for any other type of Fatou
  component. One of the tools used is a novel quasiconformal surgery
  technique of independent interest.
\end{abstract}
\maketitle

\section{Introduction}
\label{sec:intro}
Given a holomorphic endomorphism $f\colon \mathcal{S} \to \mathcal{S}$
on a Riemann surface $\mathcal{S}$, we consider the dynamical system
generated by the iterates of $f$, denoted
$f^n \defeq f\circ \cdots \circ f$, $n$ times. These classes of
dynamical systems appear naturally as complexifications of real
analytic iterative processes, but they are also of independent
interest because of their deep connections to other areas of
mathematics like analysis, topology, or number theory.  We consider two
different settings: If
$\mathcal{S} = \widehat{\C}\defeq \C\cup \{\infty\}$, then $f$ is a
rational map, while if $\mathcal{S}= \C$ and $f$ does not extend to
$\widehat{\C}$ (i.e., infinity is an essential singularity), then $f$
is a transcendental entire map.

There is a dynamically natural partition of $\mathcal{S}$ into the
\textit{Fatou set}, $F(f)$, that is, the set of points
$z\in \mathcal{S}$ where the iterates $\{f^n\}_{n\in \N}$ form a
normal family in some neighborhood of $z$, and its complement, the
\textit{Julia set} $J(f)\defeq \mathcal{S}\setminus F(f)$, which is
the locus of chaotic dynamics. A connected component of $F(f)$ is
called a \textit{Fatou component} and it is either periodic,
preperiodic, or wandering, the latter being possible only when $f$ is
transcendental. For background on the dynamics of rational and
transcendental maps, see, e.g., \cite{milnor_book, bergweiler}.

Given a point $z\in \mathcal{S}$, its {\em orbit} is the set of its
forward iterates $(f^n(z))$ for $n\in \N$, while its {\em grand orbit}
is defined as
\[
  \GO(z)\defeq \{w\in\mathcal{S} \mid f^n(z)=f^m(w) \text{ for some }
  n,m\in \N \}. 
\] 
The grand orbit of a set $U\subset \mathcal{S}$, denoted by $\GO(U)$,
is the union of all grand orbits of the points in $U$.  The
equivalence relation on $V = \GO(U)$ whose equivalence classes consist
of grand orbits is known as the {\em grand orbit relation}. We say it
is {\em discrete} if $\GO(z)$ is a discrete set for all $z\in V$;
otherwise we call it {\em indiscrete}.

It is not hard to show that, in the case of (pre)periodic Fatou
components, after removing the grand orbits of periodic points and
singular values, the nature of the grand orbit relation depends only
on the type of component considered (basin of attraction, parabolic
basin, rotation or Baker domain) \cite{McMS,FagHen06}.

However, there are examples of wandering domains with discrete, and
examples with indiscrete grand orbit relations \cite{FagHen09}. Our
goal in this paper is to further explore orbit relations in wandering
domains. More precisely, we will provide necessary and sufficient
criteria for discreteness of the grand orbit equivalence relation and,
rather strikingly, we prove the existence of wandering domains
containing subdomains, disjoint from grand orbits of singular values,
with different grand orbit relations.

For our examples, we first construct an abstract model of the desired
internal dynamics as an infinite sequence of Blaschke products. Using
quasiconformal surgery, we ``paste'' this model into an actual
wandering domain of an entire function. Since our quasiconformal
surgery technique employs some novel techniques and is of independent
interest, we state it as a separate result.

\subsection*{Motivation and statement of results}

It is frequently of interest to understand the conjugacy classes of
dynamical systems, where the type of conjugacies (conformal, smooth,
quasiconformal, topological, measurable, etc.) may depend on the
particular setting. In complex dynamics, quasiconformal conjugacies
have turned out to be particularly important. We are interested in the
deformation space (also known as \emph{moduli space}) of a given
holomorphic map $f$, rational or transcendental, which consists of the
set of maps which are quasiconformally conjugate to $f$, modulo
conformal conjugacy. As in the classical case of Riemann surfaces,
moduli spaces often have orbifold singularities, and it is frequently
more convenient to work with \emph{Teichmüller spaces}, constructed
as marked versions of moduli spaces.

In \cite{McMS}, McMullen and Sullivan developed a very general
framework of moduli and Teichmüller spaces for complex dynamical
systems. In addition, they showed that the Teichmüller space of a
rational map is a finite-dimensional complex manifold, naturally
factoring as the product of Teichmüller spaces associated to
periodic Fatou components and one factor corresponding to invariant
line fields on the Julia set.

More precisely, given a rational or entire map $f$, one considers the
set of {\em (dynamically) marked points} of $f$, defined as
\[
  \widehat{J}(f)\defeq \cl{\GO(S(f) \cup \Per(f))},
\]
where $S(f)$ denotes the set of singular values of $f$ (critical and
asymptotic values, and accumulations thereof) and $\Per(f)$ the set of
periodic points. This set is a completely invariant closed set
containing the Julia set $J(f)$, since repelling periodic points are
dense in $J(f)$. Its complement
$\widehat{F}(f) \defeq \C \setminus \widehat{J}(f)$ is thus an open
completely invariant subset of the Fatou set, and the restriction
$f|_{\widehat{F}(f)}:\widehat{F}(f) \to \widehat{F}(f)$ is a covering
map.

Each grand orbit $V$ of a connected component of $\widehat{F}(f)$
contributes a factor $\mathcal{T}(V,f)$ to the Teichmüller space
$\mathcal{T}(f)$ of $f$ in the following way: If the grand orbit
relation on $V$ is discrete, then its grand orbit space $V/f$ is a
(connected) Riemann surface, and $\mathcal{T}(V,f)$ is the
Teichmüller space of $V/f$. If the grand orbit relation is
indiscrete, then either all components of $V$ are punctured disks or
they are all annuli of finite modulus, and the closures of grand orbit
equivalence classes consist of disjoint unions of analytic Jordan
curves. If $V_0$ is any component of $V$, these curves foliate $V_0$,
and in this case $\mathcal{T}(V,f)$ is the Teichmüller space of the
foliated disk or annulus $V_0$. (This construction does not depend on
the choice of the component $V_0$.)

Lastly, there is a factor $\mathcal{T}(\widehat{J}(f),f)$ supported on
$\widehat{J}(f)$. In the case of rational maps,
$\mathcal{T}(\widehat{J}(f),f)$ corresponds to the space of measurable
invariant line fields on the Julia set, and it is conjectured that it
is always trivial except in the case of flexible Lattès maps.

In order to use this description of $\mathcal{T}(f)$ for a
transcendental entire map $f$, it is essential to understand the
nature of the grand orbit relation in the components of
$\widehat{F}(f)$. For the components contained in basins of
attraction, parabolic basins, and rotation domains, the situation is
exactly like it is for rational maps, see \cite{FagHen09}. (The only
difference is that there may be infinitely many such grand orbits, and
each one may contain infinitely many grand orbits of singular values.)

The situation becomes more interesting if $V$ is a subset of the grand
orbit of a wandering domain or of a cycle of Baker domains. In the case
of a Baker domain, the grand orbit relation is discrete, but the type
of the Riemann surface $V/f$ depends on the type of Baker domain
(hyperbolic, simply or doubly parabolic, see \cite{FagHen06,
  konig}). As pointed out above, in the case of a wandering domain, examples in
\cite{FagHen09} show that the grand orbit relation can be either
discrete or indiscrete. Given the variety of possible dynamics in
wandering domains, two natural questions arise.
\begin{enumerate}
\item[\textbf{Q1.}] Can points with discrete and indiscrete grand
  orbits coexist in a single component of $\widehat{F}(f)$?
\item[\textbf{Q2.}] Can components of $\widehat{F}(f)$ with discrete
  and indiscrete grand orbit relations coexist in a single wandering
  domain?
\end{enumerate}

Our two first results in this paper address these questions and read
as follows.

\begin{thmA}[\hypertarget{thm:A}{GO relations in $\widehat{F}(f)$}]
  Let $f$ be an entire map and let $V$ be the grand orbit of a component
  of $\widehat{F}(f)$. Then the following are equivalent:
  \begin{enumerate}[\rm (a)]
  \item \label{thmA:a} The grand orbit relation in $V$ is discrete;
  \item \label{thmA:b} There exists $z\in V$ such that $\Orb(z)$ is
    discrete.
  \item \label{thmA:c} For each $z\in V$, there exists a neighborhood
    $U$ of $z$ such that $\Orb(z')\cap U=\{z'\}$ for all $z'\in U$.
  \end{enumerate}
  If, in addition, $V$ is contained in the grand orbit of a wandering
  domain, then \eqref{thmA:a}-\eqref{thmA:c} are equivalent to
  \begin{enumerate}[\rm (a)] \setcounter{enumi}{3}
  \item \label{thmA:d} There exists a simply connected
    domain $W \subset V$ such that $f\vert_{f^n(W)}$ is injective for
    all $n\geq 0$.
  \end{enumerate}
\end{thmA}

While the previous theorem gives a negative answer to the first
question, our main result answers the second in the positive.

\begin{thmB}[\hypertarget{thm:B}{Coexistence of different GO relations
    in a wandering domain}]
  There exists a transcendental entire function $f$ with a simply
  connected wandering domain $U$ and two connected components $W_1$
  and $W_2$ of $U \cap \widehat{F}(f)$ such that $V_1=\GO(W_1)$ and
  $V_2=\GO(W_2)$ have discrete and indiscrete
  grand orbit relations, respectively.
\end{thmB}

We shall see in \Cref{sec:coexistence} that in fact the number of
components of $U \cap \widehat{F}(f)$, for a wandering domain $U$, can
be arbitrarily large, and that many combinations in terms of orbit
relations are possible.

Observe that \thmBref provides only examples of {\em simply connected}
domains. This is not a consequence of the limitations of our
technique, but rather of the non-existence of such examples for
multiply connected wandering domains. It follows from \thmAref
and the dynamics of entire functions in multiply connected wandering
domains, as described in \cite{mcwd}, that grand orbit relations in multiply
connected wandering domains are always indiscrete; see
Corollary~\ref{prop:mcwd}.

The proof of \thmBref proceeds in two steps. We first build an
abstract model of the desired internal dynamics in the orbit of the
wandering domain as a sequence of Blaschke products $(b_n:\D\to \D)$.
These maps $b_n$ reflect the dynamics of a fixed polynomial inside a
growing sequence of equipotential curves. Secondly, we show that the
sequence $(b_n)$ can be realized in a sequence of wandering domains of
a transcendental entire map, using quasiconformal surgery. Informally
speaking, we start with a given map $f$ possessing an orbit of
wandering domains, say $(U_n)$ with $f:U_n\to U_{n+1}$, and we
``replace'' $f$ inside each $U_n$ by the Blaschke product $b_n$.

This surgery procedure is new and, since it is performed in the whole
forward orbit of $U_0$ at once, it requires uniform estimates on the
quasiconformal constants of the gluing maps involved. In order to
establish these estimates, we require both sequences $(b_n)$ and
$(f|_{U_n})$ to be {\em uniformly hyperbolic}, which means that the
degrees in the sequences are uniformly bounded and that the critical
points stay within bounded hyperbolic distance from an orbit (see
Sections \ref{sec:hyp_BP} and \ref{sec:uniformlyWD}). The technique is
of independent interest and summarized in the following theorem.

\begin{thmC}[\hypertarget{thm:C}{A replacement surgery procedure}]
  Let $f$ be an entire function with an orbit of simply connected
  wandering domains $(U_n)$ and let $(b_n:\D\to\D)$ be a sequence of
  Blaschke products such that $2\leq \deg(f\vert_{U_n})=\deg(b_n)$ for
  all $n\geq 0$. Suppose both sequences $(f|_{U_n})$ and $(b_n)$ are
  uniformly hyperbolic. Then there exists an entire function $g$ with
  an orbit of simply connected wandering domains $(\widetilde{U}_n)$
  such that the following hold:
  \begin{enumerate}[\rm (a)]
  \item \label{replacement:a}
    There is a sequence of conformal maps
    $\theta_n: \widetilde{U}_n \to \D$ such that
    \[
      g\vert_{ \widetilde{U}_n} = \theta_{n+1}^{-1} \circ b_n\circ
      \theta_n.
    \]
  \item \label{replacement:b}
    There are neighborhoods $V$ and $\widetilde{V}$ of $\C
    \setminus \bigcup_{n\geq 0} U_n$ and $\C
    \setminus \bigcup_{n\geq 0} \widetilde{U}_n$, respectively, and 
    a quasiconformal map $\Phi:\C \to \C$ such that
    \[
      g\vert_{\widetilde{V}} = \Phi \circ f\vert_{V} \circ \Phi^{-1}.
    \]
    Furthermore, $\Phi(\partial U_n) = \partial\widetilde{U}_n$ for
    all $n$.
  \end{enumerate}
\end{thmC}

We remark that this surgery operation only changes the dynamics of the
original map $f$ well inside the wandering orbit $(U_n)$, so the
dynamical properties of $f$ outside this set remain the same. In
contrast, for examples constructed using approximation theory there is
almost no control on the resulting entire function outside the
prescribed domains. For more information on the use of approximation
theory to construct functions with wandering domains we refer to the
seminal paper by Eremenko and Lyubich \cite{pathological}, as well as
to several recent interesting constructions; see e.g. \cite{befrs1,
  BocThaler, mrw1, fastescaping}.

Combining \thmCref with results in \cite{befrs1}, in
\Cref{sec:uniformlyWD} we prove the following realization statement,
generalizing \cite[Cor.~1.10]{associates_20}.

\begin{thmD}[\hypertarget{thm:D}{Realization of abstract wandering
  sequences}]
  Let $(b_n)$ be a sequence of uniformly hyperbolic Blaschke
  products. Then there exists a transcendental entire function $f$
  with an orbit of simply connected wandering domains $(U_n)$ and a
  sequence of conformal maps $(\theta_n: U_n \to \D)$ such that
  \[
    f\vert_{U_n} = \theta_{n+1}^{-1} \circ  b_n\circ \theta_n.	
  \]
\end{thmD}

\begin{structure} \Cref{sec:prelim} has definitions, notation, and
  some preliminary results on holomorphic self-maps of the unit disk,
  while \Cref{sec:qcprelim} contains results on quasiconformal and
  quasisymmetric maps needed for our quasiconformal surgery.  In
  \Cref{sec:hyp_BP} we define the notion of uniformly hyperbolic
  Blaschke products and prove the main properties we will need.
  \Cref{sec:uniformlyWD} deals with uniformly hyperbolic wandering
  domains and contains the proofs of \thmCref and \thmDref.
  In \Cref{sec:go-relations} we discuss grand orbit relations and
  prove \thmAref.
  Finally, in \Cref{sec:coexistence} we study polynomials in relation
  to uniformly hyperbolic Blaschke products and give the proof of
  \thmBref.
\end{structure}

\subsection*{Acknowledgments}
We are grateful to Curtis McMullen for generously providing the idea
of mimicking polynomial dynamics in a sequence of Blaschke products in
order to obtain our example in \thmBref. We also thank Christian
Henriksen for helpful discussions. This material is partially
based upon work supported by the National Science Foundation under
Grant No.\ DMS-1928930 while the authors participated in a program
hosted by the Mathematical Sciences Research Institute in Berkeley,
California, during the Spring 2022 semester.

\section{Preliminaries}
\label{sec:prelim}
\subsection{Notation and basic definitions} We denote by $\C$ the
complex plane and by $\widehat{\C} = \C \cup \{ \infty \}$ the
extended complex plane. We denote the closure of a set $A$ (either in
$\C$ or in $\CC$) by $\overline{A}$. We write $A\subset B$ to indicate
that $A$ is a subset of $B$, including the possibility that $A=B$. In
cases where we want to stress that $A$ is a proper subset of $B$, we
use the notation $A \subsetneq B$ instead.  For sets $A, B\subset \C$,
we write $A \Subset B$ to indicate that $A$ is \emph{compactly
  contained} in $B$, i.e., $\overline{A}$ is compact and
$ \overline{A} \subset B$.  Given a set $A \subset \C$ and a constant
$c \in \C$, we write $A + c \defeq \{ z + c \colon z \in A \}$ for the
translate of the set $A$ by $c$.

For $r>0$, we denote by $\D_r$ and $\s_r$ the open disk and the circle
of radius $r$ centered at $0$, respectively. In the case $r=1$, we
write $\D \defeq \D_1$ and $\s \defeq \s_1$.  For $0<r<R$,
$\A_{r,R}\defeq \{ r<|z|<R \}$ is the round annulus with inner radius
$r$ and outer radius $R$. When $R=1$, we write $\A_r \defeq
\A_{r,1}$. We refer to the conformal modulus of a topological annulus
$A$ as $\Mod(A)$. Sequences of points, sets, and functions such as
$(x_n)$ are indexed by $n\geq 0$, unless otherwise specified. For a
simply connected domain $U\subset \C$, $d_U(\cdot, \cdot )$ denotes
the hyperbolic distance in $U$, while $\rho_U(\cdot)$ denotes the
density of the hyperbolic metric.

Given domains $U,V\subset \C$ and a holomorphic map $f\colon U\to V$,
we denote by $\Crit(f)$ the set of its critical points and by $S(f)$
the set of its singular (critical and asymptotic) values.  Moreover,
if $f$ is a proper map, we denote by $\deg(f)$ the topological degree
of $f$.

\begin{defn}
  \label{def:sequences}
  Let $U_n \subset \C$ and let $f_n\colon U_n\to U_{n+1}$ be
  holomorphic proper maps. The sequence $(f_n)$ has \emph{uniformly
    bounded degree} if there exists $d \in \N$ such that
  $\deg(f_n)\leq d<\infty$ for all $n\geq 0$. Moreover, an
  $(f_n)$-\emph{orbit} (or simply an \textit{orbit} when the context
  is clear) is a sequence of points $(z_n)$, $z_n\in U_n$, such that
  $f_n(z_n) = z_{n+1}$.
\end{defn}
 
The notions of topological and (quasi-)conformal equivalence are usually
defined for pairs of maps. In this paper, we extend this definition to
sequences of maps.
\begin{defn}\label{def:top_equiv} Let $U_n,V_n \subset \C$ and let
  $f_n\colon U_n\to U_{n+1}$ and $g_n\colon V_n\to V_{n+1}$ be holomorphic
  functions. We say that the sequences $(f_n)$ and $(g_n)$
  are \emph{topologically equivalent} if there exists a sequence of
  homeomorphisms $h_n\colon U_n\to V_n$ such that
  \begin{equation}
    \label{eq:conj}
    g_n = h_{n+1} \circ f_n \circ h_n^{-1} \quad \text{ for all }
    n \ge 0.
  \end{equation}
  If, in addition, there exists $K \ge 1$ such that all maps $h_n$ are
  $K$-quasiconformal, then we say that the sequences are
  \textit{($K$-)quasiconformally equivalent}. If all maps $h_n$ are
  conformal, the sequences are said to be \textit{conformally
    equivalent}.
\end{defn}

\begin{remark} Orbits and singular values are preserved under
  topological equivalence, that is, if $(f_n)$ and $(g_n)$ are
  sequences of maps as in Definition~\ref{def:top_equiv} that are
  topologically equivalent under a sequence of homeomorphisms $(h_n)$,
  then $(z_n)$ is an $(f_n)$-orbit if and only if $(h_n(z_n))$ is a
  $(g_n)$-orbit. Moreover, $h_n(\Crit(f_n))=\Crit(g_n)$ and
    $h_{n+1}(S(f_n))=S(g_n)$ for all $n\geq0$.
\end{remark}

If $f$ is an entire map with a wandering domain $U_0$, we will denote
by $U_{n+1}$ the Fatou component that contains $f(U_n)$, $n\geq 0$,
and say that $(U_n)$ is an \textit{orbit} of wandering domains. Note
that if $\deg(f\vert_{U_n})$ is finite, then $U_{n+1}=f(U_n)$, and $(f_n\defeq f\vert_{U_n})$ is a sequence as above.

\textit{Finite Blaschke products} are proper holomorphic self-maps of the unit disk. It is well-known
that they have a finite degree $d \ge 1$ and can be written as
\[
  b(z) =  e^{i\theta} \prod_{k=1}^{d} \frac{z-a_k}{1-\overline{a_k}z},
\]
where $\theta\in \R$ and $|a_k|<1$ for $k=1,\ldots,d$. We will frequently work with infinite sequences of Blaschke products
$(b_n)$, $n\geq 0$.

\subsection{Estimates for holomorphic self-maps of the unit disc}

For $\rho \in (0,1)$, let 
\[
  \F_\rho = \{ g: \D \to \D \text{ analytic,} \, g(0)=0, |g'(0)|=\rho \}.
\]

The following is a variation of the Schwarz Lemma. It follows from
\cite[Cor.~2.4]{befrs1}, but we include a direct proof for
completeness. 
\begin{lemma}[Uniform Schwarz Lemma]
  \label{lem:uniform-schwarz}
  Let $\rho, r \in (0,1)$. Then there exists $c \in (0,1)$,
  explicitly given by $c = \frac{\rho+r}{1+r \rho}$ such that $|g(z)|
  \le c|z|$ for all $g \in \F_\rho$ and all $|z| \le r$.
  
  In particular, if $\rho\leq r$, then
  $|g(z)| \le \frac{2r^2}{1+r^2}<r$.
\end{lemma}

 \begin{remark}
	The general statement without the explicit formula for $c$ can be
	proved using the Schwarz lemma and a normality argument. 
\end{remark}

\begin{proof}[Proof of Lemma \ref{lem:uniform-schwarz}]
Let $g \in \F_\rho$. By composing $g$ with a rotation, we may
  assume $g'(0) = \rho$. Define $h(z) = g(z)/z$, with $h(0) = g'(0)
  = \rho$. We have to show that $|h(z)| \le c = \frac{\rho+r}{1+r
    \rho}$ for $|z| \le r$.
  
  The function $h$ is an analytic map from $\D$ into itself, so by the
  Schwarz lemma it satisfies $d_{\mathbb{D}}(h(z), h(w)) \le
  d_\D(z,w)$ for all $z,w \in \D$. Fixing $|z| \le r$, we get that
  \[
    d_{\mathbb{D}}(\rho,h(z)) = d_{\mathbb{D}}(h(0),h(z)) \le
    d_{\mathbb{D}}(0,z) \le d_{\mathbb{D}}(0,r),
  \]
  and thus
  \[
    d_{\mathbb{D}}(0,h(z)) \le d_{\mathbb{D}}(0,\rho) +
    d_{\mathbb{D}}(\rho,h(z)) \le  d_\D(0,\rho) + d_\D(0,r).
  \]
  Since $d_\D(0,r) = d_\D(0,-r)$, and since the interval $(-1,1)$ is a
  hyperbolic geodesic containing $-r$, $0$, and $\rho$, we get that
  \[
    d_{\mathbb{D}}(0,\rho) + d_{\mathbb{D}}(0,r) = d_{\mathbb{D}}(-r,\rho).
  \]
  The fact that $T(w) \defeq \frac{w+r}{1+rw}$ is a hyperbolic
  isometry gives
  \[
    d_{\mathbb{D}}(-r,\rho) = d_{\mathbb{D}}(T(-r), T(\rho)) =
    d_{\mathbb{D}}\left( 0, \frac{\rho+r}{1+r\rho} \right).
  \]
  Combining all the inequalities and equalities above we get
  \[
    d_{\mathbb{D}}(0,h(z)) \le d_{\mathbb{D}}\left( 0,
      \frac{\rho+r}{1+r\rho} \right),
  \]
  so that
  \[
    |h(z)| \le \frac{\rho+r}{1+r\rho}.
  \]
  
  Finally, suppose that $0<\rho \leq r<1$. The function $\rho \mapsto
  c(\rho) = \frac{\rho + r}{1+\rho r}$ is increasing on $[0,1]$, so
  $c(\rho) \le c(r) = \frac{2r}{1+r^2}$ and the result follows
  directly.
\end{proof}

\begin{cor}
  \label{cor:uniform-schwarz-preimage}
  Assume that $0 < \rho \le r < 1$, and $g \in \F_\rho$. Then
  $g^{-1}(\overline{\D}_r)$ contains the disk $\overline{\D}_R$ with
  $R = \sqrt{\frac{r}{2-r}} \in (r,1)$.
\end{cor}
\begin{proof}
  Let $R = \sqrt{\frac{r}{2-r}}$. Then $0  < R < 1$, and
  $\frac{2R^2}{1+R^2} = r$, so applying
  Lemma~\ref{lem:uniform-schwarz} with $R$ instead of $r$ gives
  that $|g(z)| \le r < R$ whenever $|z| \le R$ and 
   $\rho \le R$, which is satisfied because $\rho \le r$ and $r
  < R$.
\end{proof}

\section{Quasiconformal and quasisymmetric maps}
\label{sec:qcprelim}
In the proof of \Cref{thm:gluingmap}, which is the basis for \thmCref,
we will require some results on quasiconformal and quasisymmetric
maps.  For a general reference, see \cite{ahl,LehVir}, and for results
in the context of holomorphic dynamics, see \cite{BraFag14}.

Let $X$ and $Y$ be connected subsets of $\C$, and let $f:X \to Y$ be a
homeomorphism.  We say that $f$ is \emph{$H$-quasisymmetric} if
$H \ge 1$ is a constant such that for $x_1, x_2, x_3 \in X$
\[
  |f(x_1)-f(x_2)| \le H |f(x_2)-f(x_3)| \qquad \text{whenever }
  |x_1-x_2| \le |x_2-x_3|.
\]
In the general theory of quasisymmetric maps on metric spaces,
homeomorphisms satisfying this condition are usually called ``weakly
$H$-quasisymmetric''. However, by
\cite[Cor.~10.22]{heinonenLecturesAnalysisMetric2001}, for every
weakly $H$-quasisymmetric homeomorphism between connected subsets of
$\C$ there exists a homeomorphism $\eta:[0,\infty) \to [0,\infty)$
depending only on $H$ such that $f$ is strongly $\eta$-quasisymmetric,
i.e.,
\[
  |f(x_1)-f(x_2)| \le \eta(t) |f(x_2)-f(x_3)| \qquad \text{whenever }
  |x_1-x_2| \le t |x_2-x_3|.
\]
With this stronger condition, it is immediate that the composition
$f_2 \circ f_1$ of strongly $\eta_k$-quasisymmetric homeomorphisms
$f_k$, $k=1,2$,  is strongly $\eta_2 \circ \eta_1$-quasisymmetric, which in turn
implies that the composition $f_2 \circ f_1$ of $H_k$-quasisymmetric
homeomorphisms $f_k$ is $H$-quasisymmetric, with $H$ only depending
on $H_1$ and $H_2$. We say that $f:X \to \C$ is an
\emph{$H$-quasisymmetric embedding} if it is an $H$-quasisymmetric
homeomorphism onto its image $f(X)$.

By \cite[Theorem~11.14]{heinonenLecturesAnalysisMetric2001}, every
$K$-quasiconformal map $f: \C \to \C$ is $H$-quasisymmetric with $H$
only depending on $K$, and every $H$-quasisymmetric $f: \C \to \C$ is
$K$-quasiconformal, with $K$ depending only on $H$. (Note that
conformal maps from a bounded onto an unbounded domain are never
quasisymmetric, so the assumption that $f$ is defined in the whole
plane is important here.)

Let $\gamma$ be a Jordan curve in $\C$. We say that $\gamma$ is a
\emph{$K$-quasicircle} if it is the image of the unit circle under a
$K$-quasiconformal self-map of $\C$. Ahlfors showed that there is a
simple geometric characterization of quasicircles \cite[Ch.~IV]{ahl}.
For distinct points $z,w \in \gamma$, let $\diam_\gamma(z,w)$ denote the
minimum of the diameters of the two components of
$\gamma \setminus \{ z,w \}$. We say that $\gamma$ is of
\emph{$M$-bounded turning} if for all distinct $z,w \in \gamma$ we
have $\diam_\gamma(z,w) \le M |z-w|$. The result of Ahlfors is that every
$K$-quasicircle is of $M$-bounded turning, with $M=M(K)$, and that
every Jordan curve of $M$-bounded turning is a $K$-quasicircle, with
$K=K(M)$.

Ahlfors and Beurling gave an explicit formula showing that every
$H$-quasisymmetric mapping $f: \R \to \R$ extends to a
$K$-quasiconformal map $F: \C \to \C$, with $K$ depending only on $H$,
see \cite[Sec.~IV]{ahl}. Combining this with the characterization of
quasicircles given above, Tukia generalized
this result to show the following:
\begin{thm}[Qc extension of quasisymmetric embeddings of circles
  \cite{tukiaExtensionQuasisymmetricLipschitz1981}]
  \label{thm:qc-extension-of-circle-embeddings}
  Let $f:S \to \C$ be an $H$-quasisymmetric embedding of a circle or a
  line $S \subset \C$. Then $f$ extends to a $K$-quasiconformal map
  $F: \C \to \C$, with $K$ depending only on $H$. In particular, the
  image $f(S)$ is a $K$-quasicircle.
\end{thm}
\begin{remark}
  In Tukia's paper, the result is stated for quasisymmetric embeddings
  of the real line. However, it is not hard to show that this result
  is true for arbitrary circles and lines (and Tukia's arguments in
  the proof show it as well.)
\end{remark}

We say that two domains $U$ and $V$ are \emph{$K$-quasiconformally
  equivalent} if there exists a $K$-quasiconformal isomorphism
$\phi: U\to V$.

\begin{thm}[{Extension of qc maps on compact sets \cite[Theorem~8.1,
    p. 96]{LehVir}}]
  \label{thm:lv-qc-extension}
  Let $\psi:D \to D'$ be a $K$-quasiconformal homeomorphism between
  domains $D, D' \subset \C$, and let $E \subset D$ be
  compact. Then, there exists $K'=K'(K,D,E)$ and a $K'$-quasiconformal
  map $\Psi: \C \to \C$ such that $\psi=\Psi$ on $E$.
\end{thm}

One immediate consequence of this theorem is that an embedding of a
circle which extends to a conformal map of an annular neighborhood is
a quasisymmetric embedding, in a quantitative way. The following
corollary is the precise version of this statement.
\begin{cor}[Quasisymmetric image of a circle]
  \label{cor:conformal-qs}
  Let $r>0$ and $0<t<1$, and let $\psi: \A_{rt,r/t} \to \C$ be a conformal
  embedding. Then $\psi|_{\s_r}: \s_r \to \C$ is an $H$-quasisymmetric
  embedding, where $H$ depends only on $t$.
\end{cor}
\begin{proof}
  By rescaling the annulus, we may assume that $r=1$. 
  \Cref{thm:lv-qc-extension} implies that $\psi|_\s$ has a
  $K$-quasiconformal extension to the plane, where $K$ depends only on
  $t$. This extension is then $H$-quasisymmetric, with $H$ again
  depending only on $t$, establishing the claim.
\end{proof}

We will also need the following interpolation result, which is a
quantitative version of \cite[Proposition~2.30(b)]{BraFag14}.

\begin{thm}[Interpolation between boundary maps of annuli]
  \label{lem:qc-interpolation}
  Let $K_0 \ge 1$ be a constant, let $\A_{r,R}$ be a round annulus of
  modulus $M = \frac1{2\pi} \log \frac{R}{r}$, and let
  $A \subset \C$ be a topological annulus whose inner and outer
  boundary are Jordan curves $\gamma_i$ and $\gamma_o$, respectively,
  with $K_0^{-1}M \le \Mod A \le K_0 M$. Furthermore, assume that
  $\phi_i$ and $\phi_o$ are orientation-preserving $H$-quasisymmetric
  homeomorphisms from the inner and outer boundary circles of
  $\A_{r,R}$ to $\gamma_i$ and $\gamma_o$, respectively.  Then there
  exists a homeomorphism 
  $\phi: \overline{\A_{r,R}} \longrightarrow \overline{A}$ which is  $K$-quasiconformal  in $\A_{r,R}$ with $K=K(H,K_0,M)$, and coincides
  with $\phi_i$ and $\phi_o$ on the respective inner and outer
  boundaries of $\A_{r,R}$.
\end{thm}

Before proving this theorem, we will first need some auxiliary
results.  In the statement of the following theorem, we say that a map
$f:X \to \C$ on an arbitrary subset $X \subset \C$ is
$K$-quasiconformal if it extends to a $K$-quasiconformal embedding of
a neighborhood of $X$. Furthermore, for disjoint Jordan curves
$S,T \subset \C$, we use the notation $S \prec T$ to indicate that
$S$ is contained in the bounded component of $\C \setminus T$, i.e.,
that $S$ is ``inside'' $T$.

\begin{thm}[{Quantitative qc extension theorem
    \cite[5.8]{tukiaLipschitzQuasiconformalApproximation1981}}]
  \label{thm:qc-annulus}
  Suppose $0<a<b<c<d$ and $K_0,M \ge 1$ are constants. Let
  $D = \A_{a,b} \cup \A_{c,d}$, and let $f_0: \overline{D} \to \C$ be a
  $K_0$-quasiconformal embedding. Assume that the images of the four
  boundary circles are nested in the same way that the circles
  themselves are, i.e., that
  $f_0(\s_a) \prec f_0(\s_b) \prec f_0(\s_c) \prec f_0(\s_d)$, and that
  $\diam f_0(\s_d) \le M \diam f_0(\s_a)$. Then there exists a
  $K$-quasiconformal map $f: \overline{\A_{a,d}} \to \C$ such that
  $f=f_0$ on $\s_a \cup \s_d$, where $K=K(a,b,c,d,K_0,M)$.
\end{thm}
\begin{rem}
  The cited theorem actually applies to quasiconformal maps in
  arbitrary dimension, in which case the constant $K$ additionally
  depends on the dimension.
\end{rem}
\begin{cor}[Interpolation between circle maps]
  \label{cor:qc-annulus-interpolation}
  Suppose $0<r_1<r_2$ and $H \ge 1$ are constants. Let $f_j:\s_{r_j}
  \to \s_{r_j}$ be $H$-quasisymmetric for $j=1,2$. Then there exists
  $K=K(r_1,r_2,H)$, and a $K$-quasiconformal map $f:\C \to \C$, such
  that $f=f_j$ on $\s_{r_j}$ for $j=1,2$.
\end{cor}
\begin{proof}
  By the Beurling-Ahlfors extension theorem \cite[Sec.~IV]{ahl}, there
  exists $K_0=K_0(H)$ such that each $f_j$ extends continuously to a
  $K_0$-quasiconformal self-map of the disk $\D_{r_j}$, normalized by
  $f_j(0) = 0$, which can then be extended by reflection to a
  $K_0$-quasiconformal map of the plane.  By Mori's theorem
  \cite[Sec.~III]{ahl}, we have the uniform H\"older estimates
  \[
    |f_j(z)-f_j(w)| \le M_j |z-w|^{1/K_0}
  \]
  whenever $z,w \in \overline{\D_{r_j}}$, with explicit constants
  $M_j = 16 r_j^{1-1/K_0}$. In particular, this shows that for any
  $\epsilon > 0$ there exists $\delta = \delta(\epsilon,r_j,K_0) > 0$
  such that $||f_j(z)|-r_j|\leq \epsilon $ whenever
  $||z|-r_j|\leq \delta$.  (This holds without the assumption that
  $|z| \le r_j$, using the fact that $f_j$ is defined by reflection
  for $|z|>r_j$.) Picking $\epsilon>0$ small enough such that
  $r_1 + \epsilon < r_2 -\epsilon$, we use the corresponding $\delta$
  (reducing it to be smaller than $\epsilon$, if necessary) to define
  the radii to use in applying \Cref{thm:qc-annulus} as $a = r_1$,
  $b = r_1 + \delta$, $c = r_2 -\delta$, and $d = r_2$. Defining
  $D = \A_{r_1,r_1+\delta} \cup \A_{r_2-\delta,r_2}$ and
  $f_0:D \to \C$ by
  \[
    f_0(z) =
    \begin{cases}
      f_1(z) & \text{for } r_1 \leq |z|<r_1+\delta, \\
      f_2(z) & \text{for } r_2-\delta < |z|\leq r_2,
    \end{cases}
  \]
  we have that $f_0$ is a $K_0$-quasiconformal embedding respecting
  the nesting of the boundary circles, and that $\diam f_0(\s_{r_2}) =
  M \diam f_0(\s_{r_1})$ with $M = r_2/r_1$. This shows that the
  assumptions of \Cref{thm:qc-annulus} are satisfied, so that
  there exists a $K$-quasiconformal map $f: \overline{\A_{r_1,r_2}}
  \to \C$ such that $f = f_0 = f_1$ on $\s_{r_1}$ and $f = f_0 = f_2$
  on $\s_{r_2}$, where $K = K(r_1, r_2, H)$. Extending $f$ by $f_1$ on
  $\D_{r_1}$, and by $f_2$ on $\C \setminus \overline{\D_{r_2}}$, we
  end up with the claimed extension to the plane.
\end{proof}

\begin{lemma}[Extension of qc maps between annuli]
  \label{lem:qc-extension}
  Let $0<r<1$, and let $\psi:\A_r \to A$ be a $K$-quasiconformal map
  onto an annulus $A$ bounded by two $K$-quasicircles. Then $\psi$ has
  a $K'$-quasiconformal extension to $\C$, with $K'$ only depending on
  $K$ and $r$.
\end{lemma}
\begin{proof}
  Let $\Gamma^i$ and $\Gamma^o$ denote the inner and outer boundaries
  of $A$ respectively.  Every $K$-quasicircle $\Gamma$ admits a
  $K^2$-quasiconformal reflection (see e.g.\ \cite[Ch.~IV, D]{ahl} or
  \cite[Prop. 4.9.12]{hubbard}), i.e. a $K^2-$quasiconformal
  involution $\phi_\Gamma$ fixing every point of $\Gamma$ and
  interchanging its complementary components.\footnote{In fact,
    something slightly stronger is true: $K$-quasicircles always admit
    $K$-quasiconformal reflections. However, a $K^2$-quasiconformal
    reflection suffices for our purposes, and the proof of this result
    is much easier.} Using the reflection with respect to the circle
  of radius $s>0$, $\tau_s(z)=s^2/\bar{z}$, for $s=1$ and $s=r$, and the two reflections
  $\phi_{\Gamma^i}$ and $\phi_{\Gamma^o}$, we can extend $\psi$ to
  $\A_{r^2,1/r}$ by
  \[
    \widehat{\psi}(z) =
    \begin{cases}
      \psi(z) & \text{if $z\in \A_r$};\\
      \phi_{\Gamma^o}\circ\psi\circ\tau_1(z) & \text{if $1\leq |z|
        < \frac1r$ }\\
      \phi_{\Gamma_i}\circ\psi\circ\tau_r(z) & \text{if
        $r^2< |z| \leq r$ }.
    \end{cases}
  \]

  Then, $\widehat{\psi}: \A_{r^2,1/r} \to \widehat{A}$ is a
  $K^3$-quasiconformal map, where $\widehat{A}$ is a topological
  annulus containing the closure of $A$. Applying
  \Cref{thm:lv-qc-extension} with $D=\A_{r^2,1/r}$ and
  $E=\overline{\A_r}$ we get the desired $K'$-quasiconformal extension
  $\Psi: \C \to \C$, with $K'$ only depending on $K$ and $r$.
\end{proof}

We are finally ready for the proof of \Cref{lem:qc-interpolation}.
\begin{proof}[Proof of \Cref{lem:qc-interpolation}]
  By scaling the round annulus, we can assume that $R=1$, so that
  $\A_{r,R} = \A_r$ and $M = \frac1{2\pi} \log \frac1r$. From the
  condition on the moduli of $\A_r$ and $A$, we know that there exists
  a $K_0-$quasiconformal isomorphism $\psi: \A_r \to A$. Note that by
  \Cref{thm:qc-extension-of-circle-embeddings}, the curves
  bounding $A$ are $K_1-$quasicircles, where $K_1$ only depends on
  $H$.  By Lemma~\ref{lem:qc-extension}, $\psi$ extends to a
  $K_2$-quasiconformal map of the plane, with $K_2$ only depending on
  $K_0, H$ and $M$. Then $\widetilde{\phi}_i = \psi^{-1} \circ \phi_i$
  and $\widetilde{\phi}_o = \psi^{-1} \circ \phi_o$ are
  orientation-preserving $H_1$-quasisymmetric self-maps of the inner
  and outer boundary circles of $\A_r$, with $H_1$ only depending on
  $K_2$ and $H$.

  By Corollary~\ref{cor:qc-annulus-interpolation}, there exists a a
  $K_3$-quasiconformal homeomorphism $\widetilde{\phi}:\A_r \to \A_r$
  with boundary maps $\widetilde{\phi}_i$ and $\widetilde{\phi}_o$,
  where $K_3$ only depends on $M$ and $H_1$.  Defining
  $\phi \defeq \psi \circ \widetilde{\phi}$, we obtain the desired
  $K$-quasiconformal map $\phi: \A_r \to \overline{A}$, with $K = K_2 \cdot K_3$
  only depending on $K_0$, $H$, and $M$.
\end{proof}

\section{Uniformly hyperbolic Blaschke products}
\label{sec:hyp_BP}

A Blaschke product $b:\D \to \D$ of degree $2\leq d <\infty$ is {\em
  hyperbolic} if it has a fixed point $z_0$ inside $\D$, which by the
Schwarz Lemma is unique. Moreover, $|b'(z_0)|<1$ and $z_0$ is a global
attractor under $b$. Note that we always consider Blaschke products
restricted to the unit disk. In particular, when speaking about the
critical points and critical values of a Blaschke product, we only
consider those inside the unit disk.

In the following, we will normalize our Blaschke products to fix zero,
so that
\begin{equation}\label{eq_BP_fix0}
  b(z) = z \prod_{k=1}^{d-1} \frac{z-a_k}{1-\overline{a_k}z},
  \qquad a_k\in\D.
\end{equation}
Then, the {\em multiplier} of the fixed point at $a_0 = 0$ is
\begin{equation}\label{eq_mult}
  b'(0)=(-1)^{d-1} \prod_{k=1}^{d-1} a_k,
\end{equation}
and therefore in particular $|b'(0)|\leq \min_k |a_k|$. 

Hyperbolic Blaschke products of the form \eqref{eq_BP_fix0} satisfy
the following properties.
\begin{lemma}[Properties of hyperbolic Blaschke products] \label{lem:blaschke-basics}
  Let $b$ be a Blaschke product of the form \eqref{eq_BP_fix0}. Then
  the critical points of $b$ in the unit disk belong to the hyperbolic
  convex hull of its zeros $\{a_0, a_1, \ldots,a_{d-1}\}$.
  Furthermore, given a radius $r$ satisfying $\max_{k} |a_k| <r<1$,
  then $D_r \defeq b^{-1}(\D_r)$ satisfies the following:
  \begin{enumerate}[\rm (a)]
  \item \label{item:jordan-domain} $D_r$ is an analytic Jordan domain
    and $\overline{\D_r}\subset D_r$;
  \item \label{item:annulus} $A^r\defeq D_r\setminus \overline{\D_r}$
    is an annulus whose modulus satisfies
    \[
      \frac{1}{4\pi} \log \frac{1}{r(2-r)} \leq \Mod(A^r)\leq
      \frac{d-1}{2\pi d} \log \frac{1}{r},
    \]
   where the right equality is satisfied  if $b(z)=z^2$.
  \item \label{item:derivative-estimate} $|b'(0)|< r^{d-1}$.
  \end{enumerate}
\end{lemma}

\begin{proof}
  The fact that the critical points of $b$ belong to the hyperbolic
  convex hull of the zeros of $b$ is a classical result by Walsh, see
  \cite[Theorem~6.1.6]{GMR}.

  Corollary~\ref{cor:uniform-schwarz-preimage} shows that we have
  $\overline{\D_r} \subset D_r$, and the assumption $r > \max_k |a_k|$ implies that all zeros of $b$ are contained in $\D_r$, hence in the
  connected component of $D_r$ containing $\D_r$. Since every
  connected component of $D_r$ must contain at least one zero, this
  shows that $D_r$ is connected, and by the maximum principle, it is
  also simply connected.  The assumption further implies that all the
  critical values of $b$ are contained in $\D_r$, thus $\partial D_r =
  b^{-1}(\s_r)$ is an analytic Jordan curve.
  
  For (\ref{item:annulus}), note that $b$ maps
  $\D\setminus \overline{D_r}$ to $\D\setminus \overline{\D_r}$ with
  degree $d$. Since the latter has modulus $-\frac{1}{2\pi}\log r$, it
  follows that
  $\Mod(\D\setminus \overline{D_r})= -\frac{1}{2\pi d}\log r$.  Now,
  by the Grötzsch principle, we have
  \[
    \Mod(A^r) + \Mod(\D\setminus\overline{D_r}) \leq \Mod
    (\D\setminus \overline{\D_r}),
  \]
  which gives the right inequality. Observe that if $b(z)=z^2$,
  $D_r=\{|z|< r^{1/d}\}$ and the modulus is directly computable.  To
  see the left inequality, let $\rho\defeq \vert b'(0)\vert$ and note
  that by our assumption on $r$, $|\rho| < r$.
  Corollary~\ref{cor:uniform-schwarz-preimage} implies that $A^r$
  contains the round annulus $\{ r< |z|<R \}$, where
  $R\defeq \sqrt{\frac{r}{2-r}}$, and the estimate follows.

  Lastly, (\ref{item:derivative-estimate}) follows directly from
  (\ref{eq_mult}).
\end{proof}

Next we consider sequences $(b_n)$ of Blaschke products. For any such
sequence, we denote by $a_0^n=0, a_1^n, \ldots, a^n_{d_n-1}$ the zeros
of $b_n$, where $d_n\defeq {\rm deg}(b_n)$. Sometimes, it will be
  convenient to think of the domains and ranges of the maps $b_n$ as
  distinct copies of the unit disk $\D^{(n)}$, where
  $b_n:\D^{(n)} \to \D^{(n+1)}$ for any $n\geq 0$. 

\begin{defn}[Uniformly hyperbolic sequence of Blaschke products]
  \label{def_uh}
  We say that a sequence $(b_n)$ of hyperbolic Blaschke products of
  the form \eqref{eq_BP_fix0} is \textit{uniformly hyperbolic} if
  $(b_n) $ has uniformly bounded degree and there exists $0<r<1$ such
  that $|a_k^n|<r$ for every $0\leq k< \deg(b_n)$ and every $n\in\N$.
\end{defn}

\begin{remark}
  Recall that, by definition, hyperbolic Blaschke products are of
  degree at least $2$, and so if $(b_n)$ is uniformly hyperbolic, then
  $2\leq \min_n \deg(b_n)\leq \max_n \deg(b_n)<\infty$.
\end{remark}

As a consequence of Lemma \ref{lem:blaschke-basics}, uniformly hyperbolic
sequences of Blaschke products have the following properties. Recall that for a sequence 
$(b_n)$ of Blaschke products, we denote by 
$B_n:\D^{(0)} \to \D^{(n)}$  the composition of the $n$
first maps; that is, $B_n\defeq b_{n-1}\circ \cdots \circ b_0$,
$n\geq 1$.

\begin{lemma}[Properties of uniformly hyperbolic sequences of Blaschke
  products]
  \label{lem:uh-basics}
  Let $(b_n)$ be a sequence of uniformly hyperbolic Blaschke products
  and let $0<r<1$ be a radius provided by Definition
  \ref{def_uh}. For $n\geq 0$, let
  \begin{equation}\label{eq_An}
    A_n\defeq A^r_n \defeq b^{-1}_n(\D_r)\setminus \D_r  \ \subset \D^{(n)}.
  \end{equation}
  Then, the following hold:
  \begin{enumerate}
  \item \label{item:Bn_to_0} $B_n(z)\to 0$ as $n\to \infty$ for all
    $z\in\D$;
  \item \label{item:An_annulus} For each $n\geq 0$, $A_n$ is an
    annulus, and there exist constants $m, M\in (0,\infty) $,
    independent of $n$, such that $m\leq \Mod A_n \leq M$.
  \item \label{item:A_n_once} For each $z\in \D$, there exists at most
    one $n\in \N$ such that $B_n(z)\in A_n$.
  \end{enumerate}
\end{lemma}

\begin{proof}
  Note that, for each $n$, $\rho_n\defeq|b_n'(0)|\leq r^{d_n-1}$ by
  Lemma \ref{lem:blaschke-basics}. Hence, $\sum_n (1-\rho_n)=\infty$ and so
  \eqref{item:Bn_to_0} holds, see \cite[Theorem 2.1]{befrs1}. Item
  \eqref{item:An_annulus} is a direct consequence of
  Lemma~\ref{lem:blaschke-basics}.  To see \eqref{item:A_n_once}, fix
  $z \in \D$ and assume that $B_n(z) \in A_n$. Then
  $B_{n+1}(z) \in b_n(A_n) \subset \D_r$. By the Schwarz lemma,
  $b_j(\D_r)\subset \D_r$ for every $j\geq 0$, hence
  $B_{n+k}(z) \in \D_r$ for all $k \ge 1$. Since
  $\D_r \cap A_n = \emptyset$, this establishes \eqref{item:A_n_once}.
\end{proof}

A priori, it is not obvious that uniform hyperbolicity is preserved
under conformal equivalence, but the following theorem provides alternative characterizations, which are more convenient in this respect.
\begin{thm}[Characterization of uniformly hyperbolic sequences of
  Blaschke products]
  \label{thm:uniform-hyp-characterization}
  Let $(b_n)$ be a sequence of Blaschke products of uniformly bounded
  degree fixing $0$. Then the following are equivalent:
  \begin{enumerate}[\rm (a)]
  \item $(b_n)$ is uniformly hyperbolic.
    \label{cond:uniform-hyp}
  \item There exists $s \in (0,1)$ such that $\Crit(b_n) \subset \D_s$
    for all $n$. \label{cond:crit-pts-1}
  \item Given any orbit $(z_n)$, there exists a constant $C$
    such that $d_{\D^{(n)}}(z_n, c) \le C$ for all $c\in \Crit(b_n)$.
    \label{cond:crit-pts-2}
  \item \label{cond:crit-pts-3}There exists an orbit
    $(z_n)$ and a constant $C$ such that $d_{\D^{(n)}}(z_n, c) \le C$ for all
    $c\in \Crit(b_n)$.
  \end{enumerate}
\end{thm}
\begin{proof}
  Equivalence of (\ref{cond:crit-pts-1}) and (\ref{cond:crit-pts-2})
  is immediate from the fact that $|z_n| \le |z_0|$ for all $n$ by the
  Schwarz lemma. The equivalence between (\ref{cond:crit-pts-2}) and
  (\ref{cond:crit-pts-3}) follows from the fact that the hyperbolic
  distance between orbits is a decreasing function of $n$.
  
  Condition (\ref{cond:uniform-hyp}) implies (\ref{cond:crit-pts-1})
  since the critical points of $b_n$ are contained in the hyperbolic
  convex hull of its zeros, see Lemma~\ref{lem:blaschke-basics}. It
  remains to prove that (\ref{cond:crit-pts-1}) implies
  ~(\ref{cond:uniform-hyp}). We will prove this claim by
  contradiction. Assume that there exists a constant $s \in (0,1)$
  such for all $n$, $\Crit(b_n)\subset \D_s$, but that the sequence
  $(b_n)$ is not uniformly hyperbolic. Let
  $c_1^n, \ldots, c_{d_n-1}^n$ denote the critical points of $b_n$
  (not necessarily distinct) while, as usual, we denote the zeros of
  $b_n$ by $a_0^n, a_1^n, \ldots, a^n_{d_n-1}$.  By passing to a
  subsequence, we may assume that
  \begin{itemize}
  \item the sequence $(b_n)$ has constant degree $d \ge 2$,
  \item the limits $\lim\limits_{n\to\infty} c_k^n = c_k$
    exist for all $k=1, \ldots, d-1$, with $|c_k| \le s$,
  \item the limits $\lim\limits_{n\to\infty} a_k^n = a_k$
    exists for all $k=1, \ldots, d-1$, with $|a_k| \le 1$, and
  \item $|a_k| = 1$ for at least one index $k$.
  \end{itemize}
  Then
  \[
    b_n(z) = z \prod_{k=1}^{d-1} \frac{z-a_k^n}{1-\overline{a_k^n}z}
    \quad \longrightarrow \quad
    b_\infty(z) = z \prod_{k=1}^{d-1} \frac{z-a_k}{1-\overline{a_k}z}
    \qquad \text{ for } n \to \infty,
  \]
  locally uniformly in $\D$. For any $k$ for which $|a_k|=1$, the
  factor $\frac{z-a_k}{1-\overline{a_k}z}$ is a constant equal to
  $-a_k$, so $b_\infty$ is a Blaschke product of degree
  $1 \le d_\infty \le d-1$. However, by Hurwitz's theorem, $b_\infty$
  has $d-1$ critical points $c_1, \ldots, c_{d-1}$ in the unit disk,
  which implies that the degree of $b_\infty$ is $d$, a contradiction.
\end{proof}
  
Using the characterization of uniformly hyperbolic Blaschke products
provided by
\Cref{thm:uniform-hyp-characterization}\eqref{cond:crit-pts-2},
we deduce that uniform hyperbolicity is preserved under quasiconformal
equivalence.

\begin{cor}[Uniform hyperbolicity is a quasiconformal invariant]
  Let $(b_n)$ and $(\widetilde{b}_n)$ be two sequences of Blaschke
  products fixing zero that are quasiconformally equivalent. Then
  $(b_n)$ is uniformly hyperbolic if and only if $(\widetilde{b}_n)$ is
  uniformly hyperbolic.
\end{cor}
\begin{proof}
  Note that it suffices to prove one direction, since the inverse of a
  $K$-quasiconformal map is $K$-quasiconformal. Suppose $(b_n)$ is
  uniformly hyperbolic. By assumption, there exists $K\geq 1$ and a
  sequence of $K$-quasiconformal maps $(h_n)$ such that
  $b_n = h_{n+1} \circ \widetilde{b}_n \circ h_n^{-1}$ for all
  $n$. Since $\Crit(\widetilde{b}_n)=h_n(\Crit(b_n))$ for all $n$ and
  the family of $K$-quasiconformal maps is equicontinuous on
  $\mathbb{D}$ with respect to the hyperbolic metric (see, for
  example, \cite[Theorem~4.4.1]{hubbard}), it follows from
  \Cref{thm:uniform-hyp-characterization}(\ref{cond:crit-pts-2})
  that $(\widetilde{b}_n)$ is uniformly hyperbolic.
\end{proof}

We conclude the section with the following theorem, which plays a
crucial role in the proof of \thmCref.  Roughly speaking, it tells us
that any two sequences of uniformly hyperbolic Blaschke products are
$K-$quasiconformally equivalent near the boundary of the unit disk, or
more precisely, in the sequence of annuli
$\D^{(n)}\setminus (\D_r \cup A_n)$, where the constant $K$ is
independent of $n$.  In the statement and proof of the following
theorem, we again use the notations $\D^{(n)}$ and
$\widetilde{\D}^{(n)}$ to denote different copies of $\D$.

\begin{thm}[$K$-quasiconformal equivalence of uniformly hyperbolic
  Blaschke products] \label{thm:gluingmap} Let
  $(b_n \colon \D^{(n)}\to \D^{(n+1)})$ and
  $(\widetilde{b}_n\colon \widetilde{\D}^{(n)}\to
  \widetilde{\D}^{(n+1)})$ be two sequences of uniformly hyperbolic
  Blaschke products with
  $\deg (b_n) =\deg(\widetilde{b}_n) \eqdef d_n$. Then there exists
  $r \in (0,1)$, $K\ge 1$ and a sequence of $K$-quasiconformal maps
  $h_n\colon \D^{(n)} \to \widetilde{\D}^{(n)}$ such that, for each
  $n\in \N$,
  \begin{enumerate}[\rm (a)]
  \item \label{item:h_n_id} $h_n(z)= z$ for $\vert z\vert\leq r$;
  \item \label{item:h_n_equiv}
    $b_n =h_{n+1}^{-1} \circ \widetilde{b}_n \circ h_n$ on the annulus
    $C_n\defeq \D^{(n)}\setminus b^{-1}_n(\D_r)$.
  \end{enumerate}
\end{thm}

\begin{proof}
  It suffices to show that the theorem holds when the maps $b_n$ are
  of the form $b_n(z) =z^{d_n}$, with
  $2\leq d_n\leq \max_n d_n\eqdef d<\infty$. Let us assume $(b_n)$ is
  of such form, and let $(\widetilde{b}_n)$ be any other sequence of
  uniformly hyperbolic Blaschke products as in the statement. For an
  illustration of the construction in the following proof, see
  Figure~\ref{fig:qc_equiv}.

  By Lemma~\ref{lem:blaschke-basics}, there exists $r \in (0,1)$ such
  that all zeros, critical points, and critical values of all $b_n$
  and $\widetilde{b}_n$ are contained in $\D_{r^2}$. Note that the
  preimage of the circle $\s_r$ under $b_n$ is the circle $\s_{r_n}$
  with radius $r_n \defeq r^{1/d_n}$, and similarly, the preimage of
  the circle $\s_{r^2}$ under $b_n$ is the circle $\s_{r_n^2}$. We
  then have the uniform bounds
  \begin{equation}
    \label{eq:uniform-bounds}
    0< r \le r^{1/2} \le r_n \le r^{1/d} < 1.
  \end{equation}
  \begin{figure}[t]
  	\centering \def\svgwidth{0.6\linewidth}
  	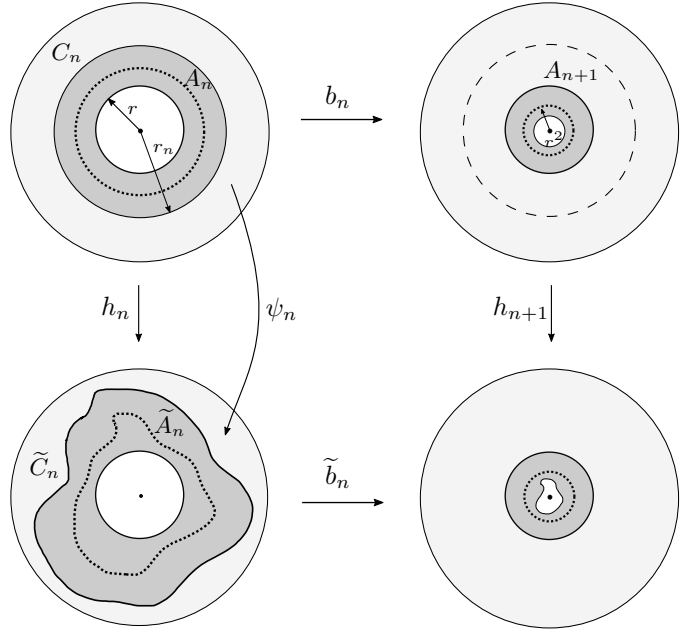
  	\caption{Illustration of the sets and maps involved in the proof
          of \Cref{thm:gluingmap}. The annulus $A_n$ and its images
          under $h_n$ and $\widetilde{b}_n$ are depicted in dark gray,
          while $C_n$ and its images are shown in light gray. The map
          $\psi_n$ sends $\A_{r^{2/d_n}}$, whose inner boundary is
          drawn with dotted lines, conformally to
          $\widetilde{b}_n^{-1}(\A_{r^2})$.}\label{fig:qc_equiv}
  \end{figure}
  With our assumption on the form of $b_n$, we have that for all $n \ge 0$,
  \[
    A_n\defeq  b_n^{-1}(\D_r) \setminus \D_r = \s_r
    \cup \A_{r, r_n},  \quad \text{and} \quad C_n = \A_{r_n}\cup \s_{r_n}.
  \]
  Likewise, let
    $\widetilde{A}_n\defeq \widetilde{b}_n^{-1}(\D_r) \setminus \D_r$,
    $\widetilde{C}_n\defeq \widetilde{\D}^{(n)}\setminus \widetilde{b}_n^{-1}(\D_r)$,  and observe that with this notation, the domain and codomain of $h_n$ can be written as the unions
    	 $$\D^{(n)}= \overline{\D}_r \cup \interior(A_n)\cup \s_{r_n}\cup \interior(C_n)  \quad \text{ and } \quad \widetilde{\D}^{(n)}=\overline{\D}_r \cup \interior(\widetilde{A}_n)\cup \widetilde{b}_n^{-1}(\s_r)\cup \interior(\widetilde{C}_n).$$
Then, for all $n\geq 0$, we define the restriction of $h_n$ to each of those subsets of $\D^{(n)}$ as follows:

\begin{enumerate}[\rm (1)]
\item \label{item_3.6_id} $h_n\colon \overline{\D}_r \to \overline{\D}_r$ is defined as the identity.

\item
\label{item_3.6_boundary} $h_n:\s_{r_n} \to \widetilde{b}_n^{-1}(\s_r)$ is defined as one of the $d_n$ maps satisfying the functional equation  
\begin{equation}\label{eq_hn}
	\widetilde{b}_n \circ h_n(z) = h_{n+1} \circ b_n(z) \left(= b_n(z) \right). 
\end{equation}
The procedure to define such a map is standard: if we choose the image of one point, say $h_n(r_n)$, to be one of the $d$ points in  $\widetilde{b}_n^{-1}(b_n(r_n))$, the map $h_n$ is defined in the whole of $\s_{r_n}$ by lifting (\ref{eq_hn}). 

\item
\label{item_3.6_qc}
{$h_n: \interior(A_n) \to \interior(\widetilde{A}_n)$} is a $K-$quasiconformal map defined by interpolation between its boundary maps defined in \eqref{item_3.6_id} and \eqref{item_3.6_boundary}, $h_n\vert_{\s_{r_n}}$ and $h_n\vert_{\s_r}$, using \Cref{lem:qc-interpolation}. 
\begin{Claim} The constant $K$ can be chosen independent of $n$.
\end{Claim}
\begin{subproof}
  To see this, we must check that the constants $H$, $K_0$, and $M$
  appearing in the assumptions of \Cref{lem:qc-interpolation} can be
  chosen uniformly, i.e., independent of $n$. From the uniform bounds
  \eqref{eq:uniform-bounds}, the annuli $\A_{r,r_n}$ have moduli
  uniformly bounded above and below, and by Lemma
  \ref{lem:blaschke-basics}, the same is true for the annuli
  $\widetilde{A}_n$. It remains to show that the inner and outer
  boundary maps are $H$-quasisymmetric, with a constant $H$
  independent of $n$. For the inner boundary maps, the identity, this
  is trivially true, so we only need to establish it for the outer
  boundary maps. To do so, we will show that $h_n|_{\s_{r_n}}$ is the
  restriction of a conformal map defined on an annular neighborhood of
  $\s_{r_n}$, and apply Corollary~\ref{cor:conformal-qs}.

  Observe that by our choice of $r$,
  \[
    b_n: b_n^{-1}(\A_{r^2})\to
    \A_{r^2} \quad \quad \text{and} \quad \quad
    \widetilde{b}_n:\widetilde{b}_n^{-1}(\A_{r^2})\to\A_{r^2}
  \]
  are analytic $d_n$-to-$1$ covering maps. Hence, the branch of
  $\psi_n \defeq \widetilde{b}_n^{-1} \circ b_n$ which agrees with $h_n$
  on $\s_{r_n}$, as chosen in \eqref{item_3.6_boundary}, is a conformal
  map from $\A_{r_n^2}$ to $\widetilde{b}_n^{-1}(\A_{r^2})$, with
  $\psi_n|_{\mathbb{S}_{r_n}}=h_n$. Again using the uniform bounds
  \eqref{eq:uniform-bounds}, we see that choosing $t=r^{1/d} \in (0,1)$,
  we have
  \[ 
    r_n^2\leq t\cdot r_n\quad \quad \text{and} \quad \quad \frac{r_n}{t} \leq 1.
  \]
  Hence, $\A_{tr_n,r_n/t} \subset \A_{r_n^2}$, which is the domain of
  $\psi_n$.  From Corollary~\ref{cor:conformal-qs} we conclude that
  $h_n = \psi_n|_{\s_{r_n}}$ is $H$-quasisymmetric, with $H$ depending
  only on $t=r^{1/d}$, which in turn is independent of~$n$.
\end{subproof}
 \item
 \label{item_3.6_commuting}
  $h_n: \interior(C_n) \to \interior(\widetilde{C}_n)$ is defined iteratively, using that $h_n$ has been defined in $\overline{\D_{r_n}}$ for all $n$: Notice that for every $n\geq 0$, the maps
 \[
 b_n: b_n^{-1}(A_{n+1})\to
 A_{n+1} \quad \quad \text{and} \quad \quad
 \widetilde{b}_n:\widetilde{b}_n^{-1}(\widetilde{A}_{n+1})\to \widetilde{A}_{n+1}
 \]
 are covering maps of degree $d_n$ between  annuli 
 (belonging to  $\D^{(n)}$ and $\D^{(n+1)}$, $\widetilde{\D}^{(n)}$
 and $\widetilde{\D}^{(n+1)}$, respectively). Hence, the $K-$quasiconformal map
 $h_{n+1}: A_{n+1}\to \widetilde{A}_{n+1}$ can be \textit{lifted} to a
 $K$-quasiconformal map
 $h_n: b_n^{-1}(A_{n+1})\to
 \widetilde{b}_n^{-1}(\widetilde{A}_{n+1})$, so that the functional equation
 \begin{equation}\label{eq_functional}
 	\widetilde{b}_n \circ h_n(z) = h_{n+1} \circ b_n(z), 
 \end{equation}
 is satisfied for every $n\geq 0$.
 
 Once this is done, we proceed in the same way to extend $h_n$
 to the next sequence of annuli
 $b_n^{-1} (b_{n+1}^{-1}
 (A_{n+2})) \subset \D^{(n)}$. It follows
 from Lemma~\ref{lem:uh-basics}\eqref{item:A_n_once} that in
 $\D^{(n)}$, these annuli are disjoint from the previous
 ones. Iterating this process infinitely many times, and since all points in $C_n$ are eventually mapped to $A_j$ for some $j\geq n$, we extend $h_n$ to the whole of $C_n$ 
 for every $n$.
\end{enumerate}

We have now defined maps
$h_n\colon \D^{(n)} \to \widetilde{\D}^{(n)}$, continuous by construction, $K$-quasiconformal for some uniform $K$ by \eqref{item_3.6_qc}, and so that \eqref{item:h_n_id} and \eqref{item:h_n_equiv} in the statement of the theorem hold by \eqref{item_3.6_id} and \eqref{item_3.6_commuting} respectively, which concludes the proof.
\end{proof}

\section{Uniform hyperbolicity in wandering domains}
\label{sec:uniformlyWD}

In \Cref{sec:hyp_BP}, we introduced the notion of uniform
hyperbolicity for sequences of Blaschke products. In this section, we
define this concept for the restriction of an entire map to an orbit
of wandering domains, discuss the relation between the two
definitions, as well as with the notion of strong
  contraction, and prove \thmCref and \thmDref.

\begin{defn}[Admissible orbit of wandering
  domains] \label{def:admissible} Let $f$ be an entire function with
  an orbit of wandering domains $(U_n)$. We say that this orbit is
  \emph{admissible} if each $U_n$ is simply connected, and the degrees
  $d_n = \deg f\vert_{U_n}$ satisfy $2 \le d_n \leq d <\infty$ for
  some $d\in \N$ and all $n$.
\end{defn}
 
Note that any entire function with an admissible orbit of wandering
domains $(U_n)$ has an \textit{associated sequence} of finite Blaschke
products $(b_n \defeq \phi_{n+1} \circ f \circ \phi_n^{-1})$, which
can be chosen to fix $0$ by considering suitable Riemann maps
($\phi_{n}$).

\begin{defn}[Uniform hyperbolicity in wandering domains]
  \label{def:unif hyp f}
  Let $f$ be an entire function with an admissible orbit of wandering
  domains $(U_n)$.  We say that $(f\vert_{U_n})$ is \textit{uniformly
    hyperbolic} if there exists an orbit $(z_n = f^n(z_0))$ and a
  constant $C$ such that $d_{U_n}(z_n, c) \le C$ for all
  $c\in \Crit(f\vert_{U_n})$ and all $n \ge 0$.
\end{defn}


As a direct consequence of \Cref{thm:uniform-hyp-characterization} and
the invariance of the hyperbolic metric under conformal maps, we have
the following.
\begin{thm}[Equivalence of definitions]
  \label{thm:unihypiff}
  Let $f$ be an entire function with an admissible orbit of wandering
  domains $(U_n)$. Let $(b_n)$ be a sequence of Blaschke products
  fixing zero associated with $(f\vert_{U_n})$. Then $(f\vert_{U_n})$
  is uniformly hyperbolic if and only if $(b_n)$ is uniformly
  hyperbolic. Furthermore, if $(f\vert_{U_n})$ is uniformly
  hyperbolic, and if $(z_n)$ is any orbit in $(U_n)$, then there
  exists a constant $C$ such that $d_{U_n}(z_n,c) \le C$ for all
  $c \in \Crit(f|_{U_n})$ and all $n \ge 0$.
\end{thm}

One possible classification of orbits of wandering domains, $(U_n)$, is proven in
\cite{befrs1}, in terms of the sequence of hyperbolic distances
$d_{U_n}(z_n,z'_n)$ between points in different orbits $(z_n)$
$(z'_n)$, giving rise to three types:
\textit{contracting}, \textit{semi-contracting} and \textit{eventually
  isometric}. One can give an even finer classification in the
contracting case according to the rate of contraction with which
iterates cluster together. More precisely, we say that $(f\vert_{U_n})$ is \textit{strongly contracting} if there
exists $c \in (0,1)$ such that
\begin{equation}\label{eq:strong}
  d_{U_n}(f^n(z),f^n(z'))=O(c^n) \text{ for one pair (and hence for
    all pairs) }z, z' \in U_0,
\end{equation}
see \cite[Definition~1.2]{befrs1}. The following proposition shows
that uniform hyperbolicity is a strictly stronger condition than
strong contraction.

\begin{prop}[Uniform hyperbolicity implies strong contraction]
  \label{contraction}
  Let $f$ be an entire function with an admissible orbit of wandering
  domains $(U_n)$. If $(f\vert_{U_n})$ is uniformly hyperbolic, then
  $(f\vert_{U_n})$ is strongly contracting. The converse is in general not true, i.e., there exist entire functions $f$ such that $(f\vert_{U_n})$ is strongly contracting, but not uniformly hyperbolic.
\end{prop}

\begin{proof}
  It is proven in \cite[Theorem~B and Lemma~3.1(b)]{befrs1} that
  strong contraction is equivalent to requiring that
  \begin{equation}\label{eq:strong-contraction}
    \limsup_{n\to\infty} \frac1n \sum_{k=1}^n |b'_k(0)| < 1,
  \end{equation}
  for any (and hence all) Blaschke sequences fixing 0 associated to
  $(U_n)$.  If we assume that $(f\vert_{U_n})$ is uniformly
  hyperbolic, then it follows from Lemma~\ref{lem:uh-basics}
  and \Cref{thm:unihypiff} that for any associated
  sequence of Blaschke products $(b_n)$ that fix $0$, there exists $r<1$
  such that $|b_n'(0)| \le r$ for all $n$, which implies that the
  average in (\ref{eq:strong-contraction}) is smaller than $r<1$.
  
 In order to construct an example for which the reverse
    implication fails, consider first the sequence of Blaschke
    products
    \[
      b_n(z)\defeq
      \begin{cases}
        z^2, & \text{for}\;n=2k-1,  k\geq 1 \\
        z \frac{z+a_n}{1+a_nz}, & \text{for}\; n=2k,  k\geq 1,
      \end{cases}
    \] 
    where $a_n \in (0,1)$, $a_2 = 1/2$, and $a_n \to 1$ for
    $n \to \infty$.  We have that $b_n(0)=0$, $b_{2k-1}'(0)=0$ and
    $b_{2k}'(0)=a_{2k} \to 1$. Let $c_n$ be the critical point of
    $b_{n}$, and note that $c_{2k} \to -1$ as $k \to \infty$.  Now let
    $D_n \defeq \{z: |z-4n|<1\}$ and let $f_n:D_n \to D_{n+1}$ be
    defined by $f_n= T_{n+1} \circ b_n \circ T_n^{-1}$, where
    $T_n(z)=z+4n$. Finally, let $B_n= b_n \circ \cdots \circ b_1$ and
    $F_n = f_n \circ \cdots \circ f_1$. We apply
    \cite[Theorem~5.3]{befrs1} to obtain a transcendental entire
    function $f$ with an orbit of simply connected wandering domains
    $(U_n)$ such that the following hold:
  \begin{enumerate}
  \item \label{wandering-1:1}
    $f(4n)= 4(n+1)$, $f^n(1/2)= F_n(1/2)$, $f(C_n)=f_n(C_n)$,
    and $f'(C_n)=f'_n(C_n)=0$ for $n \geq 1$, where $C_n=c_n+4n$. 
  \item \label{wandering-1:2}
    $d_{U_n}(f^n(0), f^n(1/2))< K_n\cdot d_{D_n}(F_n(0),
    F_n(1/2)) $ for some $K_n >1$ with $K_n \to 1$ as $n \to \infty$. 
  \item \label{wandering-1:3}
    $U_n$ is asymptotically equal to $D_n$  (as $n\to \infty$).
  \end{enumerate}

  Since $b_n(0)=0$ and $b_{2k}=z^2$, it follows by the Schwarz lemma
  that
  \[\operatorname{dist}(F_n(0), F_n(1/2))= \operatorname{dist}(0,
    F_n(1/2))< \frac{1}{2^n}.\] Since
  $d_{\mathbb{D}}(0,z)= \log (1+|z|)/(1-|z|)$ we deduce that there
  exists $1<c<2$ such that
  \[d_{D_n}(0, F_n(1/2)) = d_{\D}(0,B_n(1/2)) <\frac{1}{c^n}, \] for
  $n$ sufficiently large. It now follows by \eqref{wandering-1:3}
  together with \eqref{eq:strong} that $(f\vert_{U_n})$ is strongly contracting. Moreover, by \eqref{wandering-1:1}, $f(C_n)=f_n(C_n)$. Since $c_{2k} \to -1$ as $k \to \infty$ we
  deduce by \eqref{wandering-1:3} that $C_{2k} \to \partial
  U_{2k}$. Also, since by \eqref{wandering-1:1},
  $f(4n)=4(n+1)$, $d_{U_{2k}}(8k, C_{2k})\to \infty$. Since
  $C_{2k}$ is the critical point of $f$ in $U_{2k}$ we deduce by
  \Cref{thm:unihypiff} that $(f|_{U_n})$ is not uniformly hyperbolic.
\end{proof}

\begin{observation}[Wandering domains arising as lifts]
  \label{obs:lifts_uh}
  A well-known method for obtaining wandering domains is by
  \textit{lifting} invariant components of a self-map $F$ of
  $\C\setminus \{0\}$ by the exponential map. Namely, for any such
  $F$, there is a (non-unique) transcendental entire function $f$ such
  that $\exp\circ f=F\circ \exp$; see \cite[p.3]{befrs1} for more
  details. Suppose that $F$ has an invariant Fatou component $\Delta$
  and $f$ has an orbit of wandering domains $(U_n)$, so that
  $\exp(U_n)=\Delta$ for all $n$. If $\Delta$ is a (super)attracting
  basin, then $(f\vert_{U_n})$ is uniformly hyperbolic, whereas if
  $\Delta$ is a parabolic basin, $(f\vert_{U_n})$ is not uniformly
  hyperbolic, as it is not strongly contracting; see \cite[Corollary~3.5]{befrs1}. If $\Delta$ is a Siegel disk or a univalent Baker
  domain, $(f\vert_{U_n})$ is not uniformly hyperbolic either, since
  $\deg f\vert_{U_n} = 1$ for all $n$. Finally, if $\Delta$ is a Baker
  domain where $F$ has finite degree greater than one, then $\Delta$
  has finitely many critical points $c_j, j=1,\dots,n$, and since for
  any $z\in \Delta$, $F^n(z)\to \infty$, we have
  $d_{\Delta}(F^{n}(z),c_j) \to \infty$ for all $j$ as $n \to
  \infty$. Since the hyperbolic metric is conformally invariant, in
  particular under the inverse branches of $\exp$, it follows by
  \Cref{thm:unihypiff} that $(f\vert_{U_n})$ is not uniformly
  hyperbolic.
\end{observation}

We are now ready to prove \thmCref, using the tools developed in the
previous sections. Given an orbit of wandering domains $(U_n)$ such
that $(f\vert_{U_n)}$ is uniformly hyperbolic and a sequence of
prescribed uniformly hyperbolic Blaschke products $(b_n)$, we will
perform surgery to ``change'' the map $f$ in $(U_n)$, so that $(b_n)$
becomes an associated Blaschke sequence of the new wandering orbit.

\begin{proof}[Proof of \thmCref]
  Let $(\widetilde{b}_n)$ be a sequence of Blaschke products fixing
  zero associated with $(f\vert_{U_n})$. In particular, each
  $\widetilde{b}_n$ is of degree
  $d_n\defeq \deg(f\vert_{U_n})=\deg(b_n)$ and by
  \Cref{thm:unihypiff}, $(\widetilde{b}_n)$ is uniformly hyperbolic.
  Denote by $\varphi_n: U_n \to\D$ the Riemann maps such that
  \[
    \widetilde{b}_n\circ \varphi_n = \varphi_{n+1} \circ f, \quad \text{
      for }n\geq 0.
  \]
  By \Cref{thm:gluingmap}, applied with the roles of
  $(\widetilde{b}_n)$ and $(b_n)$ reversed, there exist $K\geq 1$,
  $0<r<1$ and a sequence of $K-$quasiconformal maps $h_n:\D\to\D$ such
  that $h_n$ is the identity on $\overline{\D_r}$ and
  \[
    \widetilde{b}_n= h_{n+1}^{-1}\circ b_n \circ h_n \quad \text{ on  }
    \quad C_n\defeq \D\setminus \widetilde{b}^{-1}_n(\D_r).
  \]
  Note that if $A_n\defeq \widetilde{b}^{-1}_n(\D_r)\setminus \D_r$, then
  $\D=\D_r \cup A_n \cup C_n$ for each $n$. 
  Define a new sequence of maps 
  \[
    \defmap{g_n}{\D}{\D}{z}{h_{n+1}^{-1} \circ b_n\circ h_n(z)}.
  \]
  By construction, $g_n$ is $K^2-$quasiregular on $\D$ but
  furthermore, for each $n\geq 0$,
  \[
    g_n=b_n   \quad \text{on $\D_r$} \quad \quad \text{and} \quad
    \quad g_n=\widetilde{b}_n \quad \text{on $C_n$.}
  \]
  Moreover, the sequences $(g_n)$ and $(b_n)$ are quasiconformally equivalent on $\D$.

  Using the sequence $(g_n)$, we modify the initial map $f$,
  by defining
  \begin{equation*}
    F(z) \defeq
    \begin{cases}
      f(z), & \text{ if } z \notin \bigcup_{n\ge 0} U_n, \\
      \varphi_{n+1}^{-1} \circ g_{n} \circ \varphi_{n}(z), & \text{ if
      } z \in U_n.
    \end{cases}
  \end{equation*}
  
  Notice that continuity of $F$ is ensured by construction, since
  $g_n(z)= \widetilde{b}_n(z)$ for all $z\in C_n$, which implies that
  $f(z)=F(z)$ on a neighborhood of $\partial U_n$.

  Hence, letting $\psi_n\defeq h_n\circ \varphi_n$ for $n\geq 0$, by
  construction, the following diagram commutes:
  \begin{figure}[htb]
    \centering
    \begin{tikzcd}
      \ldots \arrow[r,"F"] & U_n \arrow[r, "F"] \arrow[d, "\psi_n"] &
      U_{n+1} \arrow[r, "F"] \arrow[d, "\psi_{n+1}"] & \ldots \\
      \ldots \arrow[r,"b_{n-1}"] & \D \arrow[r, "b_n"]  &
      \D \arrow[r, "b_{n+1}"] & \ldots 
    \end{tikzcd}
  \end{figure}

  Note that $F$ is $K^2$-quasiregular, and holomorphic outside the
  topological annuli
  \[
    \mathcal{A}_n\defeq \varphi_n^{-1}(A_n) =
    \varphi_n^{-1}(g_n^{-1}(\D_r)\setminus \D_r) , \,\,\, n\geq 0.
  \]
  Moreover, note that for every $k\geq 0$ and $z\in U_k$, the iterates
  of $F$ satisfy
  \[
    F^n(z)=  \psi_{n+k}^{-1}\circ b_{n-1+k}\circ \cdots \circ b_k\circ
    \psi_k(z) , \quad n\in \N,
  \]
  and hence they are uniformly $K^2$-quasiregular. 

  This is enough to conclude, by Sullivan's Straightening Theorem
  \cite[Theorem~5.6]{BraFag14}, that there exists an entire function
  $g$ quasiconformally conjugate to $F$ under a global quasiconformal
  homeomorphism $\Phi:\C\to \C$, exhibiting an orbit of wandering
  domains $\widetilde{U}_n\defeq \left( \Phi(U_n)\right)$.
  Nevertheless, to show that $(b_n)$ is an associated Blaschke
  sequence of $g$ in $(\widetilde{U}_n)$, we shall look more closely
  into the construction of $\Phi$.

  We can define a complex structure $\mu$, which has bounded
  dilatation and is $F^*$-invariant. For details on the following kind of
  construction, see \cite[Sec.~1.5]{BraFag14}.
  Let $U\defeq \bigcup_{n\geq 0} U_n$ and define
  \begin{equation*}
    \mu \defeq
    \begin{cases}
      \psi_n^*(\mu_0) & \text{ on }  U_n \text{ for every } n\geq 0, 
      \\
      (F^m)^*\mu & \text{ on }  F^{-m}(U) \setminus F^{-m+1}(U)\text{ for every } m \geq 1,\\
            \mu_0 & \text{ on }  \C\setminus \bigcup_{m\ge 0}  F^{-m}(U),
    \end{cases}
  \end{equation*}
  where $\mu_0\equiv 0$. It is easy to check that 
  \begin{enumerate}
   \item  $\mu$ is $F-$invariant on $U$ by
     definition of $F$, and hence $F^*\mu=\mu$ by construction. 
 \item $\mu=\mu_0$ on $\varphi_n^{-1}(\D_r)$ for all $n$, since
   $\psi_n=\varphi_n$ is conformal on this domain;
  \item $\mu$ has dilatation at most $K$ in  $U_n$ for any $n\geq
    0$,  since the maps $\psi_n$ are uniformly
    $K-$quasiconformal. Therefore $\mu$ has dilatation at most $K$ in
    the whole complex plane, since all the pullbacks under $F$ are
    made by holomorphic maps.
  \end{enumerate}
  
  We can then apply the Measurable Riemann Mapping Theorem (see
  e.g. \cite[Theorem~5.3.2]{AIM09}, \cite[Theorem~1.28]{BraFag14}) to
  obtain a quasiconformal homeomorphism $\Phi\colon \C\to \C$ such
  that $\Phi^*\mu_0=\mu$ and an entire map
  $g(z)\defeq \Phi \circ F\circ \Phi^{-1}(z)$ such that
  $\widetilde{U}_n\defeq \Phi(U_n)$, $n\geq 0$, is an orbit of
  wandering domains.

  Now consider the maps $\theta_n\defeq \psi_n\circ
  \Phi^{-1}:\widetilde{U}_n \to \D$, which are $K-$quasiconformal. Then,
  the following diagram commutes
  \begin{figure}[htb]
    \centering
    \begin{tikzcd}
      \ldots \arrow[r,"g"] & \widetilde{U}_n \arrow[r, "g"]  &
      \widetilde{U}_{n+1} \arrow[r, "g"] & \ldots \\
      \ldots \arrow[r,"F"] & U_n \arrow[u, "\Phi"] \arrow[r, "F"] \arrow[d, "\psi_{n}"] &
      U_{n+1} \arrow[r, "F"] \arrow[d,"\psi_{n+1}"] \arrow[u, "\Phi"] & \ldots  \\
      \ldots \arrow[r, "b_{n-1}"] & \D \arrow[r, "b_n"] & \D \arrow[r,
      "b_{n+1}"] & \ldots
    \end{tikzcd}
  \end{figure}
 
  Moreover, $\theta_n^*(\mu_0) =
  (\Phi^{-1})^*((\psi_n)^*(\mu_0))=(\Phi^{-1})^*\mu=\mu_0$.  Hence, by
  Weyl's Lemma, $\theta_n$ is conformal for all $n$ and the proof is complete.
\end{proof}

We would like to use \thmCref to prove \thmDref, which claims the
existence of entire maps with an orbit of wandering domains with a
prescribed uniformly hyperbolic Blaschke sequence. In view of
\thmCref, given such a Blaschke sequence $(b_n)$ with uniformly
bounded degrees $(d_n)$, we only need to construct an entire function
$f$ such that $(f\vert_{U_n)}$ is uniformly hyperbolic in an orbit
$(U_n)$ of wandering domains with matching degrees, i.e. such that,
$\deg(f\vert_{U_n})=d_n$ for all $n$. The function $f$ is provided by
\cite[Theorem~5.3]{befrs1}, as shown in the following lemma.

\begin{lemma}[Existence of uniformly hyperbolic $(f\vert_{U_n})$ with
  prescribed degrees] \label{lem:f_with_dn}
  Let $(d_n)$ be a sequence of natural numbers with $2\leq d_n \leq d$
  for some $d\in \N$ and all $n \in \mathbb{N}$. Then there exists an
  entire function $f$ with an orbit of simply connected
  wandering domains $(U_n)$ such that $(f|_{U_n})$ is uniformly
  hyperbolic and $\deg (f|_{U_n})=d_n$ for all $n$.
\end{lemma}

\begin{proof}
Let $(b_n)$ be a sequence of uniformly hyperbolic
    Blaschke products of the form (\ref{eq_BP_fix0}) with
    $\deg(b_n)=d_n$, such that $b_n$ has $d_{n}-1$ critical points,
    say $c_n^k$, $1\leq k \leq d_{n}-1$, of multiplicity $1$ in
    $\D$. Note that such a sequence always exists, see, for example,
    \cite[Theorem~1]{zakeri}.  It follows from
    \Cref{thm:uniform-hyp-characterization}\eqref{cond:crit-pts-1}
    that there exists $s \in (0,1)$ such that $c_n^k \in \D_s$ for all
    $n, k$.
	
    Now let $D_n \defeq \{z: |z-4n|<1\}$ and let $f_n:D_n \to D_{n+1}$
    be defined by $f_n= T_{n+1} \circ b_n \circ T_n^{-1}$, where
    $T_n(z)=z+4n$. We
    apply \cite[Theorem~5.3]{befrs1} to obtain a transcendental entire
    function $f$ with an orbit of simply connected wandering domains
    $(U_n)$ such that the following hold:
    \begin{enumerate}
    \item \label{wandering-2:1} $f(4n)= 4(n+1)$, 
    and $f'(C^k_n)=0$ for
      $n \in \N, 1\leq k \leq d_n-1$, where $C_n^k= c_n^k + 4n$.
    \item \label{wandering-2:2} $\deg (f|_{U_n})=d_n$ for all
      $n \in \N$.
    \item \label{wandering-2:3} $U_n$ is asymptotically equal to $D_n$ (as $n \to \infty$).
    \end{enumerate}
    Note that by \eqref{wandering-2:1} and \eqref{wandering-2:2},
    $\Crit(f|_{U_n})=\bigcup_k C_n^k$ for every $n$. Since
    $c_n^k \in \D_s$ for all $n, k$ we deduce that
    $C_n^k \in D_{n,s}\defeq\{z: |z-4n|<s\}$. It now follows by
    \eqref{wandering-2:3} that there exists $C>0$ such that
    $d_{U_n}(4n,C_n^k) \leq C$. We conclude that $(f|_{U_n})$ is
    uniformly hyperbolic.
\end{proof}

\begin{proof}[Proof of \thmDref]
  Let $(b_n)$ be a given sequence of uniformly hyperbolic Blaschke
  products, and let $d_n\defeq \deg(b_n)$ for each $n$. Let $f$ be the
  entire map with wandering domains $(U_n)$ provided by Lemma
  \ref{lem:f_with_dn}, such that $(f\vert_{U_n})$ is uniformly
  hyperbolic and $\deg(f\vert_{U_n})=d_n$. Applying \thmCref
  to $f$ and $(b_n)$, we obtain an entire function with the desired
  properties.
\end{proof}

\section{Grand orbit relations}
\label{sec:go-relations}

In this section, we discuss properties of grand orbit relations and
prove \thmAref.  We start by giving several characterizations of
$\widehat{J}(f)$.  Let $f$ be an entire map, and define
\[
  \widehat{S}(f)\defeq \cl{\GO(S(f))}.
\]
Recall from the introduction that $\widehat{J}(f)$ denotes the closure
of the grand orbits of its singular values and periodic points. Note
that $J(f)\subset \widehat{S}(f)$, and that any (super)attracting
cycle also belongs to $\widehat{S}(f)$, since it must be contained in
the accumulation set of the orbit of some singular value. Thus, we
have that
\begin{equation}\label{jhat}
  \widehat{J}(f)= \widehat{S}(f) \cup
  \cl{\Orb(\Per(f))}=\cl{\Orb(S(f) \cup\{\text{Siegel
      periodic points}\})}.
\end{equation}

For the rest of the section, we fix an entire function $f$ and write
\[
  \widehat{J}\defeq \widehat{J}(f), \quad \text{ and } \quad
  \widehat{S}\defeq \widehat{S}(f).
\]

Before proceeding, we need some basic facts on holomorphic motions.
\begin{defn}
  \label{def:hol-motion}
  Let $U \subsetneq \C$ be a simply connected domain,
  $\lambda_0 \in U$, and $E \subset \CC$ be an arbitrary set. A
  \emph{holomorphic motion} of $E$ over $(U,\lambda_0)$ is a function
  $H: U \times E \to \CC$ such that
  \begin{enumerate}[\rm (a)]
  \item For any $z \in E$, the map $z \mapsto H(\lambda,z)$ is
    holomorphic in $U$.
  \item For any $\lambda \in U$, the map
    $\lambda \mapsto H(\lambda,z)$ is injective.
  \item $H(\lambda_0,z)=z$ for all $z \in E$.
  \end{enumerate}
\end{defn}
\begin{remark}
  Typically, holomorphic motions are defined over $\D$ with basepoint
  $\lambda_0 = 0$. However, composition (in the $\lambda$-variable)
  with a Riemann map of $U$ onto the unit disk shows that the theory
  with our slightly more general definition is completely equivalent
  to the one in the literature.
\end{remark}
Mañé, Sad, and Sullivan introduced the concept of holomorphic motions
and proved the following remarkable \emph{$\lambda$-lemma}:
\begin{thm}[{\cite[Theorem~12.1.2]{AIM09}}]
  \label{thm:lambda-lemma}
  Let $H$ be a holomorphic motion of a set $E$ over
  $(U,\lambda_0)$. Then $H$ has a unique extension to a holomorphic
  motion of the closure $\overline{E}$ (taken in $\CC$) such that
  \begin{enumerate}[\rm (a)]
  \item $(\lambda,z) \to H(\lambda,z)$ is jointly continuous in
    $U \times \overline{E}$, and
  \item For each $\lambda \in U$, the map $z \mapsto H(\lambda,z)$ is a
    quasisymmetric embedding of $\overline{E}$ into $\CC$.
  \end{enumerate}
\end{thm}
\begin{remark}
  Złodkowski later proved a stronger version of the $\lambda$-lemma,
  which says that $H$ actually extends (non-uniquely) to a holomorphic
  motion of $\CC$. For more background and proofs, see
  \cite[Ch.~12]{AIM09} or \cite[Sect.~5.2]{hubbard}.
\end{remark}

The key fact to proving \thmAref is the following lemma.

\begin{lemma}[Holomorphic motions of grand orbits] \label{holomotion}
  Let $U \subset \C\setminus \widehat{J}$ be a simply connected open
  set and $\la_0\in U$. Then there exists a unique holomorphic motion
  \[
    \begin{array}{rccl}
      H: & U  \times  \Orb(\la_0) &\longrightarrow & \chat\\
         &(\la, z) \quad \quad \ &\longmapsto &H(\la,z)\eqdef h_z(\la)
    \end{array}
  \]
  such that $H(\lambda, \cdot)$ maps $\Orb(\la_0)$ to $\Orb(\la)$,
  preserving orbit relations, i.e., if $z \in \Orb(\la_0)$ with
  $f^n(z) = f^m(\la_0)$, then $f^n(H(\lambda,z)) = f^m(\lambda)$.
\end{lemma}

\begin{proof}
  If $z\in \Orb(\la_0)$, we have that $f^n(z)=f^m(\la_0)$ for some
  $n,m\in\N$. Since $z,\la_0\notin \widehat{S}$, we know that
  $(f^n)'(z)\neq 0$, so there exists a holomorphic branch $g$
    of $f^{-n}$ in a neighborhood of $\lambda_m = f^m(\lambda_0)$ such
    that $g(\lambda_m) = z$. With this branch, we define
    $h_z = g \circ f^m$, so that $h_z$ is a holomorphic branch of
    $f^{-n} \circ f^m$, defined in a neighborhood of $\lambda_0$, with
    $h_z(\lambda_0) = z$. Now let $\gamma$ be any continuous path in
    $U$ starting at $\lambda_0$. Since $U_m = f^m(U)$ is disjoint from
    $\widehat{S}$, every point in $U_m$ is a regular value for
    $f^n$. In particular, this implies that every local holomorphic
    branch of $f^{-n}$ can be analytically continued along any
    continuous path in $U_m$. In particular, the branch $g$ of
    $f^{-n}$ has an analytic continuation along
    $\gamma_m = f^m \circ \gamma$, so $h_z = g \circ f^m$ has an
    analytic continuation along $\gamma$. Since $U$ is simply
    connected, the Monodromy Theorem then shows that $h_z$ extends to
    a holomorphic function in $U$. By the Permanence Principle, this
    analytic continuation is the unique holomorphic branch of
    $f^{-n} \circ f^m$ in $U$ satisfying $h_z(\lambda_0) = z$. It
    satisfies the functional equation
    $f^n(h_z(\lambda)) = f^m(\lambda)$ for all $\lambda \in U$, so
    this construction uniquely defines a function
    $H(\lambda,z) = h_z(\lambda)$ for
    $(\lambda,z) \in U \times \GO(\lambda_0)$ with the following
    properties:
    \begin{itemize}
    \item For each $z \in \GO(\lambda_0)$, the function $\lambda
      \mapsto H(\lambda,z)$ is holomorphic in $U$.
    \item If $f^n(z) = f^m(\lambda_0)$, then $f^n(H(\lambda,z)) =
      f^m(\lambda)$.
    \end{itemize}
    The second assertion implies in particular that for any
    $\lambda_1 \in U$ we have
    $H(\lambda, \GO(\lambda_0)) \subset \GO(\lambda_1)$. We claim that,
    in fact, equality holds. In order to show this, let
    $z \in \GO(\lambda_1)$. Then there exist $m,n \in \N$ with
    $f^n(z) = f^m(\lambda_1)$. Repeating the argument above with
    $\lambda_1$ instead of $\lambda_0$, we obtain a unique holomorphic
    branch $h_z^1$ of $f^{-n} \circ f^m$ in $U$, satisfying
    $h_z^1(\lambda_1) = z$. Defining
    $z_0 = h_z^1(\lambda_0) \in \GO(\lambda_0)$, the function $h_z^1$
    is also the unique holomorphic branch of $f^{-n} \circ f^m$ with
    $h_z^1(\lambda_0) = z_0$, so by construction we have that
    $z = h_{z_0}(\lambda_1) \in H(\lambda_1,\GO(\lambda_0))$,
    establishing the opposite inclusion $\GO(\lambda_1) \subset
    H(\lambda_1,\GO(\lambda_0))$ and hence equality
    $H(\lambda_1,\GO(\lambda_0)) = \GO(\lambda_1)$.

    We still need to show that the map $H(\la,\cdot)$ is injective for
    every $\la \in U$. Suppose to the contrary that
    $H(\lambda^*,\cdot)$ is not injective for some $\lambda^* \in
    U$. Then there exist distinct $z_1, z_2 \in \GO(\lambda_0)$ with
    $z^* = h_{z_1}(\lambda^*) = h_{z_2}(\lambda^*)$. By construction,
    $h_{z_1}$ is a holomorphic branch of $f^{-n_1} \circ f^{m_1}$ for
    some $m_1,n_1 \in \N$, so that $f^{n_1} \circ h_{z_1} =
    f^{m_1}$. Similarly, there exist $m_2, n_2 \in \N$ with
    $f^{n_2} \circ h_{z_2} = f^{m_2}$. Composing either of these
    functional equations with a suitable iterate of $f$, we may assume
    that $m_1 = m_2$, so that
    \begin{equation} \label{defining} f^{n_1}(h_{z_1}(\la)) =
      f^{n_2}(h_{z_2}(\la)) = f^m(\la)
    \end{equation}
    for all $\lambda \in U$. Since
    $h_{z_1}(\la_0) = z_1 \ne z_2 = h_{z_2}(\la_0)$, we have that
    $\lambda^*$ is an isolated solution to the equation
    $h_{z_1}(\la) = h_{z_2}(\la)$.  We now distinguish two cases.

  Suppose $n_1=n_2=n$. Then, since $h_{z_1}(\la)\neq h_{z_2}(\la)$ in
  a punctured neighborhood of $\la^*$, it follows that $z^*$ is a
  solution of multiplicity $\ge 2$ of $f^n(z)=f^m(\la^*)$ and hence
  $(f^n)'(z^*)=0$. Thus $z^*\in \widehat{S}$, which implies that
  $\la^*\in U\cap \widehat{S}$ (because $f^m(\la^*)=f^n(z^*)$), a
  contradiction.
	
  It is enough then to deal with the case $n_1>n_2$, since the reverse
  case is symmetric. In this situation, (\ref{defining}) with
  $\la=\la^*$ implies that $f^{n_2}(z^*)$ is periodic of period
  $n_1-n_2$, and hence $z^*$ is preperiodic (maybe periodic, if
  $n_2=0$). Thus, both $z^*$ and $\lambda^*$ are in $\widehat{J}$,
  leading to a contradiction.  
\end{proof}

As a first consequence, we have that either all or none of the grand orbits of points in a component of $\C\setminus \widehat{J}$ are discrete.
\begin{cor}[No coexistence of discrete and indiscrete grand orbits] \label{cor:all_or_nothing}
	Let $U$ be a component  of $\C\setminus \widehat{J}$. Given any $\lambda_0, \lambda_1\in U$,  $\Orb(\lambda_0)$ is discrete if and only if $\Orb(\lambda_1)$ is discrete.
\end{cor}
\begin{proof}
	Let $\la_0,\la_1 \in U$ and let $V\subset U$ be an open simply connected set containing $\la_0$ and $\la_1$. By Lemma \ref{holomotion}, there exists a holomorphic motion $H:V\times \Orb(\la_0) \rightarrow \C$, and $H(\la_1,\Orb(\la_0)) =\Orb(\la_1)$. 
	By the $\la-$Lemma, $H(\la_1,.)$ is a homeomorphism and hence the image of a discrete set must be a discrete set, which concludes the proof.
\end{proof}
 
In addition, we are able to prove that discreteness of orbits is equivalent to this seemingly stronger notion.
\begin{defn}[Strongly discrete GO] We say that the grand orbit of a point $z\in \C\setminus \widehat{J}$ is {\em strongly discrete} if there exists a neighborhood $U$ of $z$ such that $\Orb(z')\cap U=\{z'\}$ for every $z'\in U$.
\end{defn}

\begin{observation}\label{obs_strong_discrete} The definition of $\GO(z)$ being strongly discrete is independent of the representative of $\GO(z)$ for which such a neighborhood $U$ exists. In other words, the grand orbit of $z$ is strongly discrete if and only if for all $w\in \GO(z)$ there exists a neighborhood $U_w$ of $w$ such that $\Orb(w')\cap U_w=\{w'\}$ for every $w'\in U_w$. To see this, suppose that $\GO(z)$ is strongly discrete and let $U_z$ be a neighborhood of $z$ such that
\begin{equation}\label{eq_strong}
	\Orb(z')\cap U_z=\{z'\} \quad \text{ for every } z'\in U_z.
\end{equation}
Assume for the sake of contradiction that there is $w$ such that $f^n(w)=f^m(z)$ for some $n,m\geq 0$, and so that in any neighborhood $W$ of $w$ there are at least two elements of the same grand orbit. But for $W$ of sufficiently small diameter, $f^m(W)\subset f^n(U_z)$, and by \eqref{eq_strong}, $f^n(U_z)$ contains a unique representative of each grand orbit of the points it comprises. This means that $\deg(f^m\vert_W)\geq 2$ for any neighborhood $W$ of $w$, which contradicts the fact that since $w\notin\widehat{J}$, $f^m$ is a local homeomorphism in some neighborhood of $w$.
\end{observation}

\begin{prop}[Discrete grand orbits are strongly discrete] \label{prop_strongly_discrete}
	Let $z\in \C\setminus \widehat{J}$. The grand orbit of $z$ is discrete if and only if it is strongly discrete. 
\end{prop}

\begin{proof}
	By Observation \ref{obs_strong_discrete}, it is immediate that strongly discrete implies discrete.
	
	Now suppose that $\lambda_0\in \C\setminus \widehat{J}$ and $\Orb(\la_0)$ is discrete but not strongly discrete. Then, since $\Orb(\la_0)$ is discrete, there exists a simply connected neighborhood $U_0$ of $\la_0$ such that $\Orb(\la_0) \cap U_0=\{\la_0\}$, and two sequences $(x_n)$ and $(x_n')$ in $U_0$ such that $x_n'\in \Orb(x_n)$ for all $n\geq 0$ and $x_n,x_n' \to \la_0$ when $n\to \infty$.
	
	By Lemma \ref{holomotion} there exists a holomorphic motion $H(\la,z)$, where $H:U_0\times \Orb(\la_0) \to \C$. 
	In particular, $H(x_n, \Orb(\la_0))=\Orb(x_n)$. 
	By the $\lambda$-lemma (Theorem \ref{holomotion}),
$H$ extends to a holomorphic motion of the closure.
	
	 Hence, for every $n\geq 0$ we may choose $z_n\in \Orb(\la_0)$ such that 
	\[
	H(x_n, z_n)=x'_n.
	\]
	Observe that $z_n\notin U_0$ for any $n\geq 0$, since $\Orb(\la_0) \cap U_0=\{\la_0\}$. Since  $\chat\setminus U_0 $ is compact, there exists a convergent subsequence 
	\[
	z_{n_k} \rightarrow z_\infty \in \overline{\GO(\lambda_0)}\setminus U_0 \text{\ as $k\to\infty$}.
	\]
	But we know that $x_n\to \la_0$. Hence, as $k\to\infty$
	\[
	\la_0 \longleftarrow x'_{n_k}=H(x_{n_k},z_{n_k}) \longrightarrow H(\la_0,z_\infty) =z_\infty \notin U_0,
	\]
	where the right hand side follows from $H$ being jointly continuous and $H(\la_0,.)$ being the identity.
	This is a contradiction and hence the Proposition is proven.
\end{proof}

As a consequence, we obtain the following criterion for \textit{discreteness} of grand orbits.

	\begin{cor}[Discreteness criterion]\label{cor:discrete}  Let $z\in \C\setminus \widehat{J}$. The grand orbit of $z$ is discrete if and only if there exists a neighborhood $U$ of $z$ such that  $f^n(U)\cap f^m(U)=\emptyset$ for all $n\neq m$ and  $f\vert_{f^n(U)}$ is injective for all $n\geq 0$.
	\end{cor}

\begin{proof}
  If the grand orbit of $z$ is discrete then it is strongly discrete,
  and so there exists a neighborhood $U \subset\C\setminus \widehat{J}$ of $z$ such that every point
  is the unique representative of its grand orbit in $U$. By
  contradiction, suppose $w\in f^n(U)\cap f^m(U)$, for some
  $m>n$. This implies there exists $w'\in f^n(U)$ such that
  $f^{m-n}(w')=w.$ Note that
  $w\neq w'$, since periodic points belong to $\widehat{J}$. Hence,
  $w$ and $w'$ are in the same grand orbit, and so are their
  (distinct) preimages in $U$ under $f^n$, a contradiction. Likewise,
  if $f|_{f^n(U)}$ was not injective for some $n\geq 0$, there would
  exist two distinct points $w,w'\in f^n(U)$ mapping to the same point
  in $f^{n+1}(U)$, and hence belonging to the same grand orbit. Taking
  their preimages under $f^n$ in $U$ leads again to a contradiction.

  Now suppose that the disjointness and injectivity conditions are
  satisfied for a certain neighborhood $U$, and assume by
  contradiction that the grand orbit is not discrete. It follows by
  Proposition~\ref{prop_strongly_discrete} that the grand orbit is not
  strongly discrete, i.e., that $U$ contains two points $w,w'\in U$
  such that $f^n(w)=f^m(w')\eqdef w^*$ for some $m,n\in \N$. If $n=m$,
  this violates the injectivity of $f|_{f^k(U)}$ for some $k\leq n-1$,
  so it is not possible. If, say, $m>n$, then
  $w^*\in f^n(U)\cap f^m(U)$, a contradiction.
\end{proof}

We can now easily deduce \thmAref.
\begin{proof}[Proof of \thmAref]
  Let $V$ be the grand orbit of a component of $\widehat{F}(f)$ for an
  entire map $f$. Recall that we defined the orbit relation in $V$ to
  be indiscrete if $\GO(z)$ is an indiscrete set for some $z\in
  V$. Since, by Corollary~\ref{cor:all_or_nothing} and by
  Proposition~\ref{prop_strongly_discrete}, either the grand orbit of
  all or of none of the points in $V$ is discrete, and being discrete
  is equivalent to being strongly discrete, the equivalence between
  \eqref{thmA:a}, \eqref{thmA:b} and \eqref{thmA:c} follows. Suppose
  that, additionally, $V$ is in the grand orbit of a wandering domain
  $U$, and let $z\in V$. By passing to an iterate, we may assume that
  $z\in U$. Let $W\defeq U\cap V$, which contains the point $z$. Then,
  since $f^n(U)\cap f^m(U)=\emptyset$ for all $n\neq m$,
  $f^n(W)\cap f^m(W)=\emptyset$ for all $n\neq m$, and \eqref{thmA:b}
  and \eqref{thmA:d} are equivalent by Corollary \ref{cor:discrete}.
\end{proof}

In \cite{mcwd}, a thorough description of the dynamics of entire maps
in multiply connected wandering domains is provided. The combination
of the findings in \cite{mcwd} with \thmAref results in the following
concise proof of the fact that the orbit relation is always indiscrete
in such domains.
 
\begin{cor}[GO relations in multiply connected wandering
  domains] \label{prop:mcwd} Let $f$ be an entire map with a multiply
  connected wandering domain $U$. Then, the grand orbit relation in
  every component of $\Orb(U)\setminus \widehat{S}(f)$ is indiscrete.
\end{cor}

\begin{proof} Let $\widetilde{U}$ be a connected component of
	$U\setminus \widehat{S}$, and, for each $n\geq 1$, denote
	$U_{n+1}\defeq f(U_n)$, with $U_0=U$. Let $V\subset \widetilde{U}$
	be any simply connected domain. Then, by \cite[Theorem~1.3]{mcwd},
	there exists $N\in \N$ and a sequence of round annuli $(C_n)$ such
	that $f^n(V) \subset C_n\subset U_n$ for all $n\geq N$. Let
	$D\defeq f^N(V)$. It follows by \cite[Theorem 5.2]{mcwd} that for
	every $m\geq 0$ there exists a round annulus
	$A_m\supset C_N \supset D$ in $U_N$ such that for every $z \in A_m$
	and $m\in \N$, $f^m(z)= q_m\cdot (\phi_m(z))^{d_m},$ where $q_m>0$,
	$\phi_m$ is a conformal map on $A_m$ such that
	$\vert \phi_m(z)\vert \to \vert z\vert $ as $m \to \infty$ and
	$d_m \in \mathbb{N}$ is such that $d_m \to \infty$ as
	$m \to \infty$.
	
	Let us fix any arc, $\alpha$, of points of the same modulus
        contained in $D$. Then, for each $m$ sufficiently large, there
        exists a curve $\gamma_m\subset D$ such that
        $\phi_m(\gamma_m)=\alpha$. In addition, by increasing $m$ if
        necessary, since $d_m\to \infty$,
        $f^m(\gamma_m)=q_m (\phi_m(\gamma_m))^{d_m}$ is a circle
        traversed at least twice. This implies that
        $f\vert_{f^{N+j}(V)}$ is not injective for some
          $1\leq j\leq m$, and so by \thmAref the orbit relation in
        $\widetilde{U}$ is indiscrete.
\end{proof}

To conclude this section and our general discussion on discreteness of
grand orbit relations, we show in the following lemma how orbit
relations may vary (or be preserved) under conjugation via sequences
of homeomorphisms. To do so we place ourselves in a somewhat more
abstract setting and we consider two maps $f$ and $g$ (rational or
entire) which are topologically equivalent when restricted to some
domains $U\subset \widehat{F}(f)$ and $V\subset \widehat{F}(g)$. That
is, $g|_{f^n(V)} = h_{n+1} \circ f|_{g^n(V)} \circ h_n^{-1}$ for some
sequence $(h_n: f^n(U) \to g^n(V))$ of homeomorphisms. A particular
case is that where one of the two components, say $V$, is a periodic
Fatou component of $g$ and $f$ is its logarithmic lift, as discussed
in Observation \ref{obs:lifts_uh}. A slightly more general situation
will be particularly useful for the proof of \thmBref.

\begin{lemma}[Orbit relations under topological
  equivalence] \label{lem_new_discrete_preserved} Let $f$ and $g$ be
  entire or rational functions and suppose that
  $(f_n\defeq f\vert_{f^n(U)})$ and $(g_n\defeq g\vert_{g^n(V)})$ are
  topologically equivalent for some connected domains
  $U\subset \widehat{\C}\setminus \widehat{J}(f)$ and
  $V\subset \widehat{\C}\setminus \widehat{J}(g)$. That is, there
  exists a sequence $(h_n\colon f^n(U)\to g^n(V))$ of homeomorphisms
  such that
  \begin{equation}
    \label{eq:conj_top}
    g_n = h_{n+1} \circ f_n \circ h_n^{-1} \quad \text{ for all }
    n \ge 0.
  \end{equation}
  Then one of the following holds:
  \begin{enumerate}
  \item \label{item_no_rot} Neither $U$ nor $V$ belongs to a (preimage
    of a) rotation domain. Then the grand orbit relation in $\GO(U)$
    is discrete if and only if the grand orbit relation in $\GO(V)$ is
    discrete.
  \item \label{item_rotation} Both $U$ and $V$ belong to (preimages
    of) rotation domains. Then the grand orbit relations in $\GO(U)$
    and $\GO(V)$ are indiscrete.
  \item \label{item_one_of_each} Only one of $U$ or $V$, say $V$,
    belongs to a (preimage of a) rotation domain. Then the grand orbit
    relation is indiscrete in $\GO(V)$ and discrete in $\GO(U)$.
\end{enumerate}
\end{lemma}
\begin{proof} We may assume without loss of generality that both $U$
  and $V$ are in periodic (not preperiodic) components or in wandering
  domains.
	
  Suppose that neither $U$ nor $V$ belongs to a rotation domain. In
  order to show \eqref{item_no_rot}, by symmetry, it suffices to prove
  one direction. Suppose the grand orbit relation in $\GO(U)$ is
  discrete. Choose any point $z\in U$. Then, by \Cref{cor:discrete}
  and \thmAref, there exists a neighborhood $U_z$ of $z$ such that
  $f^n(U_z)\cap f^m(U_z)=\emptyset$ for all $n\neq m$ and so that
  $f\vert_{f^n(U_z)}$ is injective for all $n\geq 0$. Let
  $w\defeq h_0(z)$ and $V_w\defeq h_0(U_z)$. Note that
  $g^n(V_w)=h_{n}(f^n(h^{-1}_0(V_w)))$ for each $n$, and that
  $g^n\vert_{V_w}$ is injective as composition of injective
  maps. Then, by \Cref{cor:discrete} and \thmAref, the grand
  orbit relation in $\GO(V_w)$ is discrete if and only if there exists
  a neighborhood $V'_w\subset V_w$ of $w$ such that the sets
  $g^n(V'_w)$, $n\geq 0$, are pairwise disjoint. Note that if $V$ is
  in a wandering domain, then the sets $g^n(V_w)$, $n\geq 0$, are
  necessarily pairwise disjoint and we are done. Otherwise $V$ belongs
  to a periodic Fatou component of $g$. The only periodic components
  with indiscrete grand orbit relations are rotation domains and
  superattracting basins (see \cite[Section~6]{McMS} and
  \cite[Section~4]{FagHen09}). The first ones are discarded by
  hypothesis, while in a superattracting basin, the injectivity of
  $g^n|_W$ fails if $n$ is large enough, for every open set
  $W\subset V$, a contradiction.

  In order to show \eqref{item_rotation} and \eqref{item_one_of_each},
  we may assume without loss of generality that $V$ is in a rotation
  domain (and not a preimage of it). Then the orbit relation in
  $\GO(V)$ is indiscrete (see \cite[Section~6]{McMS} or
  \cite[Section~4]{FagHen09}) and $g_n$ is injective in $g^n(V)$ for
  all $n\geq 0$. If $U$ is also in a rotation domain, the same
  conclusion applies, and \eqref{item_rotation} follows. Otherwise,
  note that by \eqref{eq:conj_top}, for each $n\geq 1$, $f_n$ is
  injective as a composition of injective maps. Arguing as before, if
  $U$ is in a wandering domain, then the sets
  $f^n(U)=h_n^{-1}(g^n(h_0(U)))$, $n\geq 0$, belong to different Fatou
  components, and so are pairwise disjoint, hence the grand orbit
  relation is discrete. By injectivity of $f_n$, $U$ might also be
  part of an attracting or parabolic basin, or a Baker domain, but not
  part of a superattracting basin. Hence we conclude that the grand
  orbit relation in $\GO(U)$ must be discrete.
\end{proof}

\section{Coexistence of different orbit relations in wandering
  domains}
\label{sec:coexistence}

The goal of this section is to prove \thmBref, i.e., to find an
entire function $f$ with a wandering domain $U$ such that there are
connected components of $U\setminus \widehat{S}(f)$ both with discrete
and indiscrete grand orbit relations. We will do so by proving that
for any polynomial $P$, we can construct an entire function $f$ with
an orbit of wandering domains $(U_n)$, so that $f\vert_{U_n}$
resembles $P|_{D_n}$ for all $n$, where $(D_n)$ is an increasing
sequence of topological discs.  Finally, we will use Lemma
\ref{lem_new_discrete_preserved} to transfer the grand orbit relation
properties of $P$ to those of $f$ in the wandering domains.

Given a polynomial $P$ of degree $d\geq 2$, by Böttcher's
  theorem (see, e.g., \cite[\S 9]{milnor_book}) there exists a
  conformal map $\phi$ in a neighborhood of $\infty$ with
  $\phi(\infty) = \infty$, conjugating $P$ to $w \mapsto w^d$, i.e.,
  satisfying the functional equation $\phi(w^d) = P(\phi(w))$. Fix
  $R>1$ such that the domain of $\phi$ contains
  $\chat \setminus \D_R$ and, for each $n \in \N$, define
\begin{equation} \label{eq_poly_R}
  R_n\defeq R^{d^n}, \quad \quad
  \gamma_n \defeq \phi(\s_{R_n}) \quad \text{ and } \quad
  D_n\defeq{\rm int}(\gamma_n),
\end{equation}
where int$(\gamma_n)$ denotes the bounded connected component of
$\C\setminus \gamma_n$. Note that for all $n \ge 0$ we have that
$\overline{D_n} \subset D_{n+1}$, all critical points of $P$ are
contained in $D_n$, and the restriction $P|_{D_n}: D_n \to D_{n+1}$ is
a proper map of degree $d$.
 Note that we can view $(P|_{D_n})$ as a sequence of proper holomorphic
  maps in the sense of \Cref{def:sequences}.

\begin{figure}[t]
  \centering
  \def\svgwidth{0.8\linewidth}
  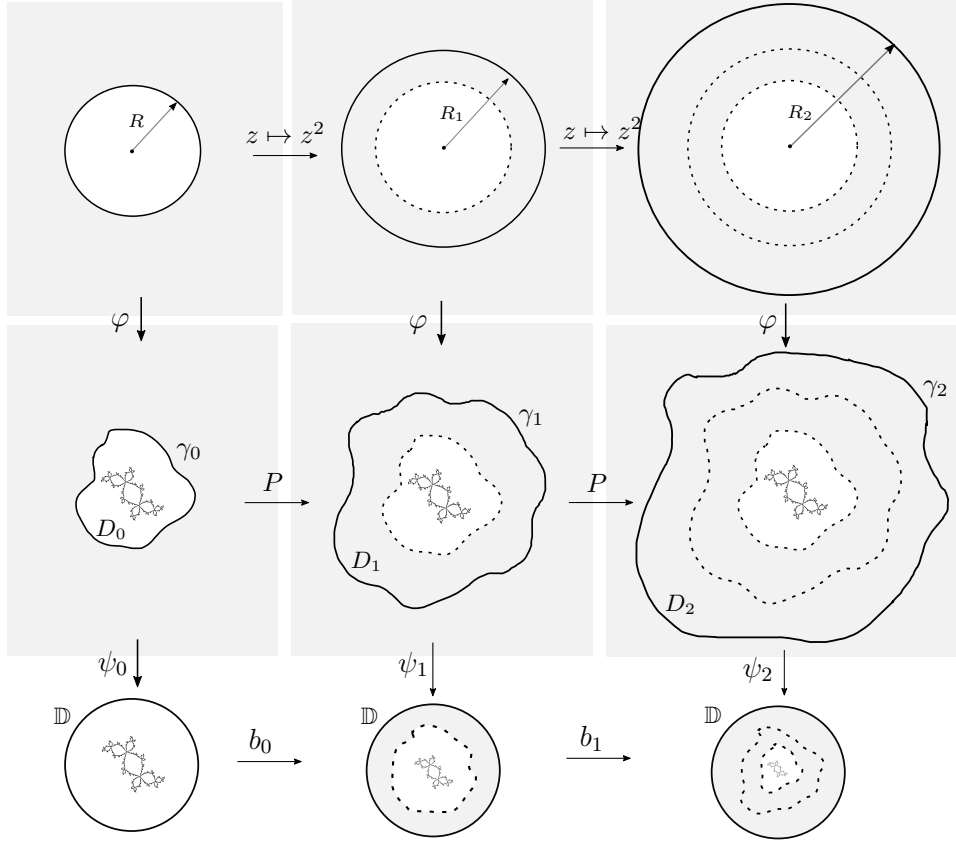
  \caption{Illustration of the construction of the sets $(D_n)$
      and the associated sequence of Blaschke products $(b_n)$ for the
      polynomial $P(z)\defeq z^2+c$, with $c=-0.123+0.745i$. (Douady's
    rabbit)}
  \label{fig:poly_to_BP}
\end{figure}

\begin{prop}[The prescribed Blaschke sequence]\label{prop_McMullen}
  Let $P$ be a polynomial of degree $d\geq 2$ and let $(D_n)$ be
  defined as in (\ref{eq_poly_R}). Then, there exists a sequence
  $(b_n)$ of uniformly hyperbolic Blaschke products of degree $d$
  conformally equivalent to the sequence $(P\vert_{D_n})$.
\end{prop}

  \begin{remark}
    See \Cref{fig:poly_to_BP} for an illustration of the construction
    of the sets $D_n$ and the associated Blaschke sequence $(b_n)$.
  \end{remark}
	
\begin{proof}
  By conjugating with an affine transformation, we may
  assume that $P(0)=0$. For each $n\geq 0$,
  choose a conformal map $\psi_n\colon D_n \to \D$, normalized
  so that $\psi_n(0)=0,$ and define the degree $d$ Blaschke product
  \begin{equation}\label{eq_defn_bn}
    b_n(z) \defeq \psi_{n+1} \circ  P  \circ \psi_n^{-1} (z), \quad
    \text{ for } z\in \D.
  \end{equation}

  We claim that the sequence $(b_n)$ is uniformly hyperbolic. Indeed,
  by \Cref{thm:uniform-hyp-characterization}\eqref{cond:crit-pts-3},
  it suffices to show that there exists a constant $C$ such that
  $d_\D(0,c) \le C$ for all $c\in \Crit(b_n)$ and all $n$. Recall that
  $\Crit(P)\subset D_0$ and that $\psi_n(\Crit(P))=\Crit(b_n) $ for
  all $n$. Since $\overline{D}_n\subset D_{n+1}$, we have
  $\rho_{D_n}(z)>\rho_{D_{n+1}}(z)$ for all $z\in D_n$ the by
  Schwarz-Pick lemma (see \cite[Theorem~10.5]{beardon_minda}), while
  the maps $\psi_n$ preserve hyperbolic distances. In particular, for
  any $c\in \Crit(P)$,
  \[
    d_\D(0,\psi_n(c))=d_{D_n}(0,
    c)\geq d_{D_{n+1}}(0,c)=d_\D(0,\psi_{n+1}(c)).
  \]
  The claim now follows with $C\defeq \max_{z\in \Crit(b_0)}d_\D(0,z)$.
\end{proof}

We aim to use \thmCref to insert the sequence of Blaschke products
$(b_n)$ obtained in Proposition \ref{prop_McMullen} inside an orbit of
wandering domains $(U_n)$ of a function $f$ (the {\em template}
function) such that $\deg(f\vert_{U_n})=d$.  Since in this case the
sequence $\deg(f\vert_{U_n})$ is constant, $f$ may be chosen to be a
\text{lift} of a self-map of $\C\setminus \{0\}$, as explained in
Observation~\ref{obs:lifts_uh}. Using this procedure and unlike in
Lemma \ref{lem:f_with_dn}, where $f$ is obtained via approximation
techniques, the singular set $S(f)$ of $f$ is known. The following is
an appropriate choice of $f$, a reformulation of
\cite[Proposition~4.1]{associates_20}, in a version that fits our
purposes.

\begin{prop}[The template function]\label{prop_associates} Let
  $d \ge 2$. Then there exists an entire function $g$ with no finite
  asymptotic values and an orbit of simply connected wandering domains
  $\widehat{U} = (U_n)_{n\geq 0}$ such that
  \begin{enumerate}[\rm (a)]
  \item \label{assoc:a} $(g|_{U_n})$ is uniformly hyperbolic, with
    $\deg g|_{U_n}=d$ for all $n\geq 0$;
  \item \label{assoc:b} each $U_n$ contains exactly one critical
    point $c_n$ of multiplicity $d-1$, and $g(c_n) = c_{n+1}$;
  \item \label{assoc:c} there are no critical points in
    $\GO(\widehat{U}) \setminus \widehat{U}$;
  \end{enumerate}
\end{prop}

\begin{proof}
  Following \cite[proof of Proposition 4.1]{associates_20}, consider
  the functions
  \begin{align*}
    q(z)&\defeq 2\cdot \sum_{j=1}^{d-1} {d-1 \choose j} \cdot
          \frac{z^j}{j}, \quad \text{and}\\ 
    G(z)&\defeq -z^2\cdot \exp (q(z)-c),
  \end{align*}
  for $c\defeq q(-1)$. The function $G$ has two superattracting fixed
  points at $-1$ and $0$, no further critical points, and a single
  finite asymptotic value at $0$. Moreover, the Fatou
    component $W$ containing $-1$ is simply connected, and the
    restriction $G\colon W\to W$ is conjugate to $z \mapsto z^d$ on
    the unit disk.
  Since $G^{-1}(0)=\{0\}$, we can lift $G$ under the
  exponential map to obtain a map $g$, defined by
  \begin{equation*}
    g(w)\defeq 2w+q(\exp(w))-c+\pi i.
  \end{equation*}
  Since the only finite asymptotic value of $G$ is the fixed
    point at zero, which is an omitted value for the
        exponential,
    the lift $g$ does not have any finite
    asymptotic values. For each $k \in \Z$, let $W_k$ be the connected component of
    $\exp^{-1}(W)$ containing $z_k = (2k-1)\pi i$. Each $W_k$ is a
    Fatou component of $g$ (see, e.g., \cite{Bergweiler_1995}), and
    since $W$ is simply connected with $0 \notin W$, each $W_k$ is
    simply connected, and is mapped conformally onto $W$ by the
    exponential function. In particular, this implies that the $W_k$
    are all distinct, i.e., $W_k \ne W_l$ for $k \ne l$. Furthermore,
    since $g(z_k) = z_{2k}$, we also know that $g(W_k) = W_{2k}$, and
    that $\deg g|_{W_{2k}} = d$, with exactly one critical point $z_k$
    of multiplicity $d-1$, mapping to the critical point
    $z_{2k}$. This shows that $W_0$ is a fixed superattracting basin
    for $g$, whereas for every odd $m \in \Z$, the sequence
    $\widehat{W}_m = (W_{2^n m})_{n\geq0}$ is an orbit of wandering
    domains, with $\widehat{W}_m \cap \widehat{W}_l = \emptyset$ for
    $m \ne l$.  Since all the $W_k$ are simply connected, and
    $\deg g|_{W_k} = d \ge 2$, all of these orbits of wandering
    domains are admissible. Furthermore, since the critical points in
    each $\widehat{W}_m$ form a single orbit, $g$ is uniformly
    hyperbolic on each of them. Finally, choosing
    $\widehat{U} = \widehat{W}_1$, i.e., $U_n = W_{2^n}$ for
    $n \ge 0$, assertion \eqref{assoc:c} follows from the fact that
    all the critical point of $g$ are contained in the disjoint union
    $\bigcup_{m \text{ odd}} \widehat{W}_m$ of orbits of wandering
    domains.
\end{proof}

In the following theorem, we again adopt the notation $\D^{(n)}$ to
denote different copies of $\D$.
\begin{thm}[Inserting the Blaschke sequence]\label{thm:poly} Let $P$
  be a polynomial of degree $d\geq 2$, and let $(D_n)$ be defined as
  in \eqref{eq_poly_R}.  Then there exists an entire function $f$ with
  an orbit of simply connected wandering domains $(U_n)$ such that
  $(f\vert_{U_n})$ and $(P\vert_{D_n})$ are
  conformally equivalent under a sequence of conformal maps $h_n:U_n
  \to D_n$. Furthermore, there is a map $V \mapsto V^*$ from
  components of  $\widehat{F}(P) \cap D_0$ to components of 
  $\widehat{F}(f) \cap U_0$ such that
  \begin{enumerate}[\rm (a)]
    \item \label{poly:b}
      $V \subset h_0(V^*)$ for every component $V$ of $\widehat{F}(P) \cap D_0$,
    \item \label{poly:c}
      for every component $W$ of $\widehat{F}(f) \cap U_0$ there
      exists a component $V$ of $\widehat{F}(P) \cap D_0$ with $V^* =
      W$, and
    \item \label{poly:d}
      if $V_1$ and $V_2$ are different components of $\widehat{F}(P)
      \cap D_0$ with $V_1^* = V_2^*$, then $V_1$ and $V_2$ are
      contained in a superattracting basin of $P$.
  \end{enumerate}
\end{thm}

\begin{remark}
  Item \eqref{poly:b} is equivalent to the simple assertion that
  $\widehat{F}(P) \cap D_0 \subset h_0(\widehat{F}(f) \cap U_0)$, and
  assertions \eqref{poly:c} and \eqref{poly:d} say in a somewhat more
  precise way that the map $V \mapsto V^*$ from connected components
  of $\widehat{F}(P) \cap D_0$ to $\widehat{F}(f) \cap U_0$ induced by
  inclusion via $h_0^{-1}$ is onto, and that it is a one-to-one
  correspondence away from superattracting basins of $P$.
\end{remark}

  \begin{proof}
  By \Cref{prop_McMullen}, there exists a sequence of Blaschke
  products $(b_n)_{n\geq 0}$ conformally equivalent to
  $(P|_{D_n})_{n\geq 0}$, under a sequence of conformal maps
    $\psi_n:\D^{(n)} \to D_n$. Let $g$ be the entire function from
  \Cref{prop_associates} and $\widehat{V} = (V_n)_{n\geq 0}$ its
  orbit of wandering domains. By \thmCref (switching the
  roles of $f$ and $g$), we can replace the dynamics of $(g|_{V_n})$
  by the Blaschke sequence $(b_n)$, obtaining an entire map $f$ 
  with an orbit of wandering domains $\widehat{U} = (U_n)$
      such that $(f|_{U_n})$ is conformally equivalent to $(b_n)$
  under a sequence of conformal maps $(\Theta_n)$ and hence also
    to $(P|_{D_n})$ under $(h_n\defeq\Theta_n\circ \psi_n)$. For an illustration of the construction,
      see the commutative diagram in \Cref{fig:comm-diag}.

    \begin{figure}[t]
      \centering
      \begin{tikzcd}
        U_0 \arrow[dd,bend right ,swap,"h_0"] \arrow[r,"f"]
        \arrow[d,"\Theta_0"] & \ldots \arrow[r,"f"] &
        U_n \arrow[r, "f"]  \arrow[d,"\Theta_n"] &
        U_{n+1} \arrow[r, "f"] \arrow[d, "\theta_{n+1}"] & \ldots \\
        \D^{(0)} \arrow[r,"b_0"] \arrow[d,"\psi_0"] & \ldots
        \arrow[r,"b_{n-1}"] & \D^{(n)}  \arrow[r, "b_n"]
        \arrow[d, "\psi_{n}"] & \D^{(n+1)} \arrow[r, "b_{n+1}"]
        \arrow[d,"\psi_{n+1}"] & \ldots  \\
        D_0 \arrow[r,"P"] &\ldots \arrow[r, "P"] &
        D_{n} \arrow[r, "P"] & D_{n+1} \arrow[r,"P"] & \ldots
      \end{tikzcd}
      \caption{Commutative diagram of the maps in the
          proof of \Cref{thm:poly}.}
      \label{fig:comm-diag}
    \end{figure}
    
   Since $g$ does not have any finite asymptotic values,
      by \thmCref the map $f$ does not have any finite
      asymptotic values either. The only critical points of $g$ in
      $\Orb(\widehat{V})$ are the critical points of multiplicity
      $d-1$ in each $V_n$. After the surgery, our new map $f$ still
      has $d-1$ critical points in each $U_n$ (counted with
      multiplicity), given by $h_n^{-1}(\Crit(P))$, and no critical
      points in $\Orb(\widehat{U}) \setminus \widehat{U}$.

      Let $S_0 = \Orb(S(f)) \cap U_0$. Since $f$ does not have any
      asymptotic values, and all the critical points in the grand
      orbit of $U_0$ are contained in the sequence $(U_n)$, we have
      that $z_0 \in S_0$ iff there exist $k,m,n \ge 0$ and a critical
      point $c_m \in U_m$ such that $f^{n}(z_0) = f^{k}(c_m)$. Since
      $f^n(z_0) \in U_n$ and $f^k(c_m) \in U_{m+k}$, this is only
      possible if $n = m+k$ and thus $f^{m+k}(z_0) =
      f^{k}(c_m)$. Writing $c = h_m(c_m)$ and $w = h_0(z_0)$, this
      shows that $z_0 = h_0^{-1}(w) \in S_0$ iff there exist
      $k,m \ge 0$ and $c \in \Crit(P)$ such that
      $P^{m+k}(w) = P^k(c)$. Since $c \in D_0$, any solution
      $w \in \C$ of this equation (not assuming $w \in D_0$ a priori)
      satisfies $w \in P^{-(m+k)}(D_k) = P^{-m}(D_0) \subset D_0$,
      since $P^{-1}(D_0) \subset D_0$ and $P^{-1}(D_{n+1}) = D_n$ for all $n \ge 0$.  This shows that
      $S_0 = h_0^{-1}(A)$ for the set
      \[
        A  = \{ w \in \C : P^{m+k}(w)=P^k(c) \text{ for some } m,k \ge  
        0 \text{ and } c\in\Crit(P)\} \subset D_0,
      \]
      and thus
      \begin{equation}\label{keyclaim}
        \widehat{J}(f) \cap U_0 = \widehat{S}(f) \cap U_0 =
        \overline{S_0} =  h_0^{-1} (\overline{A})
        \subset h_0^{-1}(\widehat{J}(P) \cap D_0),
      \end{equation}
      where the first equality follows from the fact that $U_0$ does
      not contain any periodic points of $f$ and the last equality
      follows from the fact that 
      $\overline{A} \subset \widehat{S}(P) \subset
      \widehat{J}(P)$. Taking complements and applying $h_0$ to both
      sides we conclude that $\widehat{F}(P) \cap D_0 \subset
      h_0(\widehat{F}(f) \cap U_0)$, establishing \eqref{poly:b}.

      Since $A$ contains the backward orbit of all critical points, we
      have that $J(P)\subset \overline{A}$, so
      \[
        \widehat{J}(f) \cap U_0 = h_0^{-1}(\overline{A}) \supset
        h_0^{-1}(J(P)).
      \]
      Taking complements and applying $h_0$ yields
      $h_0(\widehat{F}(f) \cap U_0) \subset F(P) \cap D_0 \subset
      F(P)$, which shows that every component of
      $h_0(\widehat{F}(f) \cap U_0)$ is contained in some Fatou
      component of $P$.

      We now treat the different types of Fatou components separately.
      If $G$ is a Fatou component of $P$ associated with an
      attracting or parabolic periodic cycle, then
      $G \cap \widehat{S}(P)$ is a countable set, so
      $V = G \setminus \widehat{S}(P)$ is connected. Let $V^*$ be the
      connected component of $\widehat{F}(f) \cap U_0$ such that
      $V \subset h_0(V^*)$. Then $h_0(V^*) \subset G$, with
      $G \setminus h_0(V^*)$ countable. This shows that $G$ contains
      exactly one component $V$ of $\widehat{F}(P) \cap D_0$, and
      $h_0^{-1}(G)$ contains exactly one component $V^*$ of
      $\widehat{F}(f) \cap U_0$, establishing claims \eqref{poly:c}
      and \eqref{poly:d} for the case of components contained in
      attracting or parabolic basins.

      If $G$ is a Fatou component of $P$ associated with a cycle of
      Siegel disks, we claim that
      $\overline{A} \cap G = \widehat{S}(P) \cap G$. We know that
      $\overline{A} \subset \widehat{S}(P)$, so we only need to
      establish the reverse inclusion. Since
      $\widehat{S}(P) = \overline{\Orb(S(P))}$, it is enough to show
      that $\Orb(S(P)) \cap G \subset \overline{A}$. Let
      $z \in \Orb(S(P)) \cap G$. Then there exist $m,n \ge 0$ and
      $c \in \Crit(P)$ with $w^* = P^m(z) = P^n(c)$. By adding an
      integer constant to both $m$ and $n$, we may assume that $w^*$
      is in the cycle of Siegel disks. Then there exists a sequence of
      integers $k_j \to \infty$ such that $P^{k_j}(w) \to w$ uniformly
      in a neighborhood of $w^*$. In particular, there exists a
      sequence $w_j \to w^*$ such that $P^{k_j}(w_j) = w^*$. Since
      $P^m$ is an open mapping, there exists a sequence $z_j \to z$
      such that $P^{m}(z_j) = w_j$, so that
      $P^{m+k_j}(z_j) = P^{k_j}(w_j) = w^* = P^n(c)$. We may assume
      that $m+k_j \ge n$ for all $j$, so that $z_j \in A$ for all $j$,
      and hence $z \in \overline{A}$, as claimed.

      The claim directly implies that
      $h_0(\widehat{S}(f) \cap U_0) \cap G = \widehat{S}(P) \cap
      G$. We know that $\widehat{J}(f) \cap U_0 = \widehat{S}(f) \cap U_0$ and that
      $(\widehat{J}(P) \cap G) \setminus (\widehat{S}(P) \cap G)$ contains
      at most one point (if $G$ contains a periodic point which is not
      in $\widehat{S}(P)$), so for every component $V$ of
      $\widehat{F}(P) \cap G$ there exists a unique component $V^*$ of
      $\widehat{F}(f) \cap U_0$ such that either $V = h_0(V^*)$, or
      $V = h_0(V^*) \setminus \{z_0\}$ for some $z_0 \in G$. This
      establishes \eqref{poly:c} and \eqref{poly:d} in the case of
      Fatou components associated with Siegel disks.

      Lastly, if $G$ is a Fatou component of $P$ associated with a
      superattracting basin (including the case where $G$ is the
      basin of infinity), \eqref{poly:d} is vacuously true, and
      \eqref{poly:c} follows easily from the fact that
      $\widehat{F}(P)$ is open and dense in $F(P)$.

\end{proof}
 
\thmBref now follows readily, from the simplest possible
  instance of \Cref{thm:poly}.

\begin{proof}[Proof of \thmBref]
Let $P$ be a quadratic polynomial given by
    $P(z) \defeq \lambda z + z^2$ with $0<|\lambda|<1$, so that $P$
    has an attracting fixed point at $0$ and its Julia set is a
    quasicircle.  Applying \Cref{thm:poly} for this choice of $P$ and
    a sequence of domains $(D_n)$ as in \eqref{eq_poly_R}, we obtain a
    corresponding transcendental entire function $f$ with an orbit of
    simply connected wandering domains $(U_n)$. The set
    $\widehat{F}(P) \cap D_0$ has two connected components $V_1$ and
    $V_2$, where $V_1$ is the attracting basin of zero with infinitely
    many points removed, while $V_2$ is the intersection of the
    superattracting basin of infinity with $D_0$. The grand orbit
    relation of $P$ is discrete in $\GO(V_1)$ and indiscrete in
    $\GO(V_2)$, see \cite[Section~6]{McMS}. Let $V_1^*$ and $V_2^*$ be
    the corresponding components of $\widehat{F}(f) \cap U_0$ given by
    \Cref{thm:poly}\eqref{poly:b}.  Applying
    \Cref{lem_new_discrete_preserved} to $P$ and $f$, restricted to
    the orbits of $V_k$ and $V_k^*$, respectively, then shows that the
    grand orbit relation for $f$ is discrete in $\GO(V_1^*)$ and
    indiscrete in $\GO(V_2^*)$, concluding our proof.
  
\end{proof}

\begin{remark}
  For the sake of simplicity, in the proof of \thmBref we chose a
  quadratic polynomial with an attracting fixed point. The choice of
  other polynomials, possibly of higher degree, with different
  combinations of types of Fatou components would lead by
  \Cref{thm:poly} to the existence of wandering domains that decompose
  into different collections of components with the two types of grand
  orbit relations, once the closure of the singular set is removed.
\end{remark}

\bibliographystyle{amsalpha}
\bibliography{wandering,wandering-zotero}
\end{document}